\documentclass[a4paper, twoside, reqno]{amsart}

\usepackage{amsmath,amssymb,bbm}
\usepackage{graphicx}

\usepackage{booktabs}
\usepackage{multirow}
\usepackage{siunitx}
\usepackage{enumerate}
\usepackage{hyperref}

\theoremstyle{plain}
\newtheorem{definition}{Definition}[section]
\newtheorem{lemma}{Lemma}[section]
\newtheorem{theorem}{Theorem}[section]

\theoremstyle{remark}
\newtheorem{remark}{Remark}[section]

\newcommand{\vertiii}[1]{{\left\vert\kern-0.25ex\left\vert\kern-0.25ex\left\vert #1 \right\vert\kern-0.25ex\right\vert\kern-0.25ex\right\vert}}

\begin{document}
	
	\pagestyle{headings}
	
	\title[A Convergent CNG Scheme for the BO Equation]{A Convergent Crank--Nicolson Galerkin Scheme for the Benjamin--Ono Equation}
	
	\author{Sondre Tesdal Galtung}
	
	\email{sondre.galtung@ntnu.no}
	
	\address{Department of Mathematical Sciences, NTNU Norwegian University of Science and Technology, NO-7491 Trondheim, Norway}
	
	\date{October 11, 2017}
	
	
	
	\begin{abstract}
		In this paper we prove the convergence of a Crank--Nicolson type Galerkin finite element scheme for the initial value problem associated to the Benjamin--Ono equation.
		The proof is based on a recent result for a similar discrete scheme for the Korteweg--de Vries equation and utilizes a local smoothing effect to bound the $H^{1/2}$-norm of the approximations locally.
		This enables us to show that the scheme converges strongly in $L^{2}(0,T;L^{2}_{\text{loc}}(\mathbb{R}))$ to a weak solution of the equation for initial data in $L^{2}(\mathbb{R})$ and some $T > 0$.
		Finally we illustrate the method with some numerical examples.
	\end{abstract}
	
	\maketitle
	
	\section{Introduction}
	\label{introduction}
	In this paper we consider a fully discrete Crank--Nicolson Galerkin scheme for the Cauchy problem associated to the Benjamin--Ono (BO) equation, which reads
	\begin{equation}
	\begin{cases}
	u_{t}+\left(\frac{u^2}{2}\right)_{x}-\mathcal{H}u_{xx} = 0, & (x,t) \in \mathbb{R} \times (0, T), \\
	u(x,0) = u_{0}(x), & x \in \mathbb{R},
	\end{cases}
	\label{BO}
	\end{equation}
	where $T > 0$ is fixed, $u : \mathbb{R} \times [0,T) \to \mathbb{R}$ is the unknown, $u_{0}$ is the initial data and $\mathcal{H}$ is the Hilbert transform defined by
	\begin{equation*}
	\mathcal{H}u(x,\cdot) := \text{p.v.} \frac{1}{\pi} \int_{\mathbb{R}} \frac{u(x-y,\cdot)}{y}\,dy,
	\end{equation*}
	where p.v. denotes the Cauchy principal value.
	
	The equation was derived independently by Benjamin \cite{Ben67} and Ono \cite{Ono75}, and serves as a model equation for weakly nonlinear long waves with weak nonlocal dispersion.
	There has been done much work on the well-posedness of \eqref{BO} and improvement of regularity restrictions on the initial data, we mention \cite{Iorio86},\cite{Abdelouhab1989} and \cite{Tao04}, where in the latter global well-posedness for initial data in $H^{s}$ for $s \ge 1$ was proved using a gauge transform resembling the famous Cole--Hopf transform for the viscous Burgers equation.
	By refining this transform, \cite{Ionescu07} extended the result to $s \ge 0$.
	
	The BO equation is formally completely integrable \cite{Ablowitz83, Kaup1998}, a property shared by the well-known Korteweg--de Vries (KdV) equation. The integrability is closely related to the fact that the BO equation admits an infinite number of conserved quantities \cite{Abdelouhab1989, Kaup1998} and a Lax pair \cite{Bock1979}, hence it can be formulated as a Hamiltonian system.
	Another feature of the BO equation is the existence of families of explicit solutions called solitons \cite{Albert1987}, and these localized, solitary wave solutions are a consequence of a delicate balance between dispersion and nonlinear convection.
	Having to approximate the simultaneous appearance of the two aforementioned effects makes the task of finding reliable numerical methods for the BO equation rather challenging, and we mention some of the various methods which have been proposed.
	In \cite{Thomee1998} the authors present a finite difference scheme for which they prove error estimates for smooth solutions of the periodic version of the BO equation, while in \cite{Dutta2015operator} they consider operator splitting methods of Godunov and Strang type for \eqref{BO} and prove corresponding $L^{2}$-convergence rates given sufficiently regular initial data. On the other hand, in \cite{Pelloni2000, Deng2009} the authors consider Fourier spectral methods and prove error estimates for sufficiently regular solutions of the periodic equation. A high order hybrid finite element-spectral scheme for the Benjamin equation, for which the BO equation is a special case, is presented in \cite{Dougalis2015} where the authors give experimental convergence rates for the method.
	
	However, we emphasize that the goal of this paper is \emph{not} to present an efficient, high-order numerical method, but rather to formulate a discrete scheme for which one is able to prove existence of a sequence of approximations that converge locally to a weak solution of the BO equation for low regularity initial data, namely functions in $L^{2}(\mathbb{R})$.
	Thus, the scheme differs from the previously mentioned papers in that it can be used as a constructive proof of the existence of solutions to \eqref{BO}.
	In this respect, the paper at hand is more in the spirit of \cite{Dutta2015difference} where the authors present finite difference schemes for \eqref{BO} and its periodic counterpart which are proved to converge to to classical solutions given initial data in $H^{2}$.
	
	We will here consider a Crank--Nicolson type Galerkin scheme for finding weak solutions to the BO equation based on a method for the KdV equation due to Dutta and Risebro \cite{Dutta2016} which can be seen as a generalization to higher order temporal approximations of \cite{Koley2015}, where the implicit Euler method is used instead of Crank--Nicolson for the temporal discretization.
	The motivation for using these rather simple time integrators is that they are easier to analyze compared to multi-step integrators such as higher order Runge--Kutta methods.
	Our strategy for establishing convergence will follow closely that of the above papers, in particular by using a local smoothing effect inherent to the equation, but some more work is required here to treat the dispersive term which contrary to the case of KdV is nonlocal for the BO equation.
	In fact, it is exactly this added challenging aspect of the nonlocal dispersion which is our main motivation for applying the methods of \cite{Koley2015, Dutta2016} to \eqref{BO}.
	The smoothing effect of the BO equation is also weaker than that of KdV, which combined with the nonlocal nature of the Hilbert transform makes it natural to consider fractional Sobolev spaces, hence making our estimates more involved than in the case of KdV.
	We note that the results presented in the current paper is in full based on \cite{master2016}.
	For a treatment of convergence rates for the scheme discussed here given sufficiently regular initial data the reader is referred to \cite{proceedings2016}.
	
	The paper is structured as follows. In Section \ref{preliminaries} we establish some preliminary technical results, e.g.\ the local smoothing effect mentioned above for a semidiscretized weak formulation of \eqref{BO}.
	The fully discrete scheme is presented in Section \ref{scheme}, where we prove existence and uniqueness of its solutions for each time step.
	The main part of the paper is contained in Section \ref{convergence}, where we first prove that the solutions of the fully discrete scheme in Section \ref{scheme} exhibit the local smoothing effect from Section \ref{preliminaries}. Then we move on to our main result, Theorem \ref{theorem_convergence}, which establishes the existence of a sequence of approximate functions which converge locally in $L^{2}$ to a weak solution of \eqref{BO}.
	Finally, in Section \ref{experiments} some numerical experiments are presented to illustrate the discrete scheme.
	
	\section{Preliminary estimates}
	\label{preliminaries}
	In the following we will give a brief explanation of our strategy.
	Let us momentarily define a weak solution to the Benjamin--Ono equation \eqref{BO} to be a function $u(x,t)$ such that $u \in C^{1}\left(\left[0,\infty\right);H^{2}(\mathbb{R})\right)$ which for all $v \in H^{2}(\mathbb{R})$ satisfies
	\begin{equation}
	\left< u_{t},v \right> + \left< \left(\frac{u^2}{2}\right)_{x}, v \right> + \left< \mathcal{H}u_{x}, v_{x} \right> = 0,
	\label{weak_1}
	\end{equation}
	where $\left< \cdot, \cdot \right>$ denotes the standard $L^{2}$-inner product.
	We now discretize the equation in time using a Crank--Nicolson method. Let $\Delta t$ be the time step size, $u^{n} \approx u(\cdot, n\Delta t)$ and $u^{n+1/2} :=  (u^{n+1}+u^{n})/2$. For a given $u^{0} \in L^{2}(\mathbb{R})$, define $u^{n}$ to be the solution of
	\begin{equation}
	\left< u^{n+1},v \right> + \Delta t \left< \left(\frac{\left(u^{n+1/2}\right)^2}{2}\right)_{x}, v \right> + \Delta t \left< \mathcal{H}\left(u^{n+1/2}\right)_{x}, v_{x} \right> = \left< u^{n}, v \right>,
	\label{CN_eqn}
	\end{equation}
	for all $v \in H^{2}(\mathbb{R})$ and $n \ge 0$.
	Assuming that the above equation has a unique solution $u^{n+1}$ we may choose $v = u^{n+1}+u^{n}$ in \eqref{CN_eqn} which yields
	\begin{equation}
	\|u^{n+1}\|_{L^{2}(\mathbb{R})} = \|u^{n}\|_{L^{2}(\mathbb{R})} = \|u^{0}\|_{L^{2}(\mathbb{R})}.
	\label{bound_L2_1}
	\end{equation}
	Here we have used that the Hilbert transform is antisymmetric, which is a consequence of the following lemma.
	\begin{lemma}
		\label{Hilbert_props}
		The Hilbert transform is a linear operator with the following properties:
		\begin{enumerate}[(i)]
			\item \emph{(Skew-symmetricity)} Assume $f \in L^{p}(\mathbb{R})$ for $1 < p < \infty$ and $g \in L^{q}(\mathbb{R})$ where $1/p + 1/q = 1$. Then we have
			\begin{equation*}
			\left< \mathcal{H}f, g \right> = -\left< f, \mathcal{H}g \right>.
			\end{equation*}
			\item \emph{(Commutes with differentiation)} For a differentiable function $f$ we have
			\begin{equation*}
			(\mathcal{H}f)_{x} = \mathcal{H}f_{x}.
			\end{equation*}
			\item \emph{($L^{2}$-isometry)} The transform preserves the $L^{2}$-norm,
			\begin{equation*}
			\|\mathcal{H}f\|_{L^{2}(\mathbb{R})} = \|f\|_{L^{2}(\mathbb{R})}.
			\end{equation*}
		\end{enumerate}
	\end{lemma}
	\begin{remark}
		Setting $p = 2$ and $g = f$ in \emph{(i)} in Lemma \ref{Hilbert_props} gives $\left< \mathcal{H}f, f \right> = 0$, which is the aforementioned antisymmetry property.
		Also, combining \emph{(ii)} and \emph{(iii)} shows that the transform is in fact isometric for all Sobolev norms $H^{k}$ with $k \in \{0\} \cup \mathbb{N}$.
	\end{remark}
	These properties will be essential in our convergence analysis and for a proof of the above lemma we refer to \cite[p.\ 317]{FourierAnalysis}.
	Note that throughout this paper $C$ will denote various positive constants which exact value is of no importance to the arguments.
	Likewise, $C(R)$ will denote various positive constants depending on the parameter $R$ and so on.
	
	Now, in \cite{Dutta2016} they introduce a smooth, positive and non-decreasing cut-off function $\varphi$ and use integration by parts to derive a local smoothing effect which bounds $u^{n}$ locally in $H^{1}$-norm.
	This technique is originally due to Kato \cite{Kato1983}, and is in fact a consequence of the commutator identity
	\begin{equation*}
	[ -\partial^{3}_{x}, \varphi ] = -3 \partial_{x} \varphi_{x} \partial_{x} - \varphi_{xxx},
	\end{equation*}
	where for two operators $A$ and $B$ we introduce the commutator bracket $[\cdot,\cdot]$ as $[A,B] := AB-BA$.
	Such identities have been generalized for the hierarchy of generalized Benjamin--Ono equations by Ginibre and Velo \cite{Ginibre1989, Ginibre1991}, and the identity relevant for us is found in \cite[p.\ 227]{Ginibre1989} and reads
	\begin{equation}
	-[\mathcal{H}(-\partial_{xx}), \varphi] = 2 D^{1/2} \varphi_{x} D^{1/2} + R_{1/2}(\varphi),
	\label{commutator_1}
	\end{equation}
	where $D^{\beta}$, $\beta > 0$ denotes the homogeneous fractional derivative defined by
	\begin{equation*}
	D^{\beta} f(x) = (-\partial_{xx})^{\beta/2} f(x) := \mathcal{F}^{-1}[|\xi|^{\beta} \hat{f}(\xi)](x),
	\end{equation*}
	and $R_{1/2}(\varphi)$ is some remainder operator.
	Here and in the following we use standard notation for the Fourier transform, defined by
	\begin{equation*}
	\mathcal{F}[f](\xi) = \hat{f}(\xi) := \frac{1}{\sqrt{2\pi}} \int_{\mathbb{R}} f(x) e^{-ix\xi}\,dx
	\end{equation*}
	for a suitable function $f$.
	Note the sign change on the left-hand side of \eqref{commutator_1} compared to that of Ginibre and Velo which is due to their use of the Hilbert transform with opposite sign.
	According to equation (40) in \cite[p.\ 228]{Ginibre1989} we have
	\begin{equation}
	\vertiii{R_{1/2}(\varphi)} \le \frac{1}{\sqrt{2\pi}} \|\widehat{D^{1}\varphi_{x}}\|_{L^{1}(\mathbb{R})} = \frac{1}{\sqrt{2\pi}} \|\widehat{\varphi_{xx}}\|_{L^{1}(\mathbb{R})},
	\label{commutator_est_1}
	\end{equation}
	where $\vertiii{\cdot}$ denotes the operator norm in $L^{2}(\mathbb{R})$.
	Using \eqref{commutator_1} and \eqref{commutator_est_1} we are able to bound the $H^{1/2}$-norm of $u^{n}$ locally, where $H^{s}(\mathbb{R})$ is the real ordered Sobolev space of functions $u$ such that $(1+|\xi|^{2})^{s/2}\hat{u} \in L^{2}(\mathbb{R})$ with the corresponding norm $\|u\|_{H^{s}(\mathbb{R})} = \|(1+|\xi|^{2})^{s/2}\hat{u}\|_{L^{2}(\mathbb{R})}$.
	
	Let us define a smooth cut-off function $\varphi \in C^{\infty}(\mathbb{R})$ satisfying:
	\begin{enumerate}[(a)]
		\item $1 \le \varphi(x) \le 2 + 2R$,
		\item $\varphi_{x}(x) = 1$ for $|x| < R$,
		\item $\varphi_{x}(x) = 0$ for $|x| \ge R+1$,
		\item $0 \le \varphi_{x}(x) \le 1$ for $x \in \mathbb{R}$ and
		\item $\sqrt{\varphi_{x}} \in C^{\infty}_{c}(\mathbb{R})$.
	\end{enumerate}
	Properties (a)--(d) are easily achievable by standard mollifier methods. The existence of functions satisfying property (e) can be motivated as follows. If this were not the case we could have started by defining a nonnegative function $h$ such that $h(x) = 1$ for $|x| < R$, $h(x) = 0$ for $|x| \ge R+1$ and $0 \le h(x) \le 1$. Then $h^{2} \in C^{\infty}_{c}(\mathbb{R})$ has the same properties as $h$ and by defining $\varphi(x) := 1 + \int_{-\infty}^{x} h(x)^{2}\,dx$ we obtain a function satisfying (a)--(d), and where $\sqrt{\varphi_{x}}$ is smooth and compactly supported by definition.
	Due to the properties of $\varphi$, $v = \varphi u^{n+1/2}$ is also an admissible test function in $H^{2}(\mathbb{R})$, and we will write $w := u^{n+1/2}$ to save space. Inserting this in \eqref{CN_eqn} and using the identity
	\begin{equation*}
	\left< \left(\frac{w^2}{2}\right)_{x}, w \varphi \right> = -\frac{1}{3} \int_{\mathbb{R}} w^{3} \varphi_{x}\,dx
	\end{equation*}
	which is easily attained from integration by parts, we have
	\begin{equation}
	\frac{1}{2} \left\|u^{n+1} \sqrt{\varphi} \right\|_{L^{2}(\mathbb{R})}^{2} + \Delta t \int_{\mathbb{R}} \mathcal{H}w_{x} (\varphi w)_{x}\,dx -\frac{\Delta t}{3} \int_{\mathbb{R}} w^{3} \varphi_{x}\,dx = \frac{1}{2} \left\|u^{n} \sqrt{\varphi} \right\|_{L^{2}(\mathbb{R})}^{2}.
	\label{CN_phi}
	\end{equation}
	We rewrite the second term on the left-hand side as
	\begin{align*}
	\begin{split}
	\int_{\mathbb{R}} \mathcal{H}w_{x} (\varphi w)_{x}\,dx &= \int_{\mathbb{R}} \varphi w \mathcal{H}(-\partial_{xx})w \,dx \\
	&= \frac{1}{2} \left<w, \varphi \mathcal{H}(-\partial_{xx}) w\right> - \frac{1}{2} \left<w, \mathcal{H}(-\partial_{xx}) \varphi w\right> \\
	&= \frac{1}{2} \left< w, -[\mathcal{H}(-\partial_{xx}), \varphi] w \right> \\
	&= \left< D^{1/2}w, \varphi_{x} D^{1/2} w \right> + \frac{1}{2} \left< w, R_{1/2}(\varphi) w \right>,
	\end{split}
	\end{align*}
	where we have used \eqref{commutator_1}.
	To show that the last term is bounded we use the following fact:
	Given $f \in C^{N}(\mathbb{R})$ and $f^{(k)} \in L^{1}(\mathbb{R})$ for $0 \le k \le N$ we have
	\begin{equation}
	|\hat{f}(\xi)| \le \frac{C}{(1+|\xi|)^{N}},
	\label{fourierBounded}
	\end{equation}
	for some suitable C depending on $N$ and $\|f^{(k)}\|_{L^{1}(\mathbb{R})}$ for $0 \le k \le N$. This estimate is fairly standard and easily obtained using properties of Fourier transforms of differentiated functions, see e.g.\ \cite[p. 109]{FourierAnalysis}, and the straightforward inequality $\|\widehat{f^{(k)}}\|_{L^{\infty}(\mathbb{R})} \le \|f^{(k)}\|_{L^{1}(\mathbb{R})}/\sqrt{2\pi}$. As $\varphi_{xx}$ belongs to $C^{\infty}_{c}(\mathbb{R})$ and in particular $C^{2}_{c}(\mathbb{R})$, we have $\varphi^{(2+k)} \in L^{1}(\mathbb{R})$ for $k=0,1,2$. 
	According to \eqref{fourierBounded} we then have $|\widehat{\varphi_{xx}}(\xi)| \le \frac{C}{(1+|\xi|)^{2}}$ and thus $\|\widehat{\varphi_{xx}}\|_{L^{1}(\mathbb{R})} \le 2C$.
	Then it follows from \eqref{commutator_est_1} that
	\begin{equation*}
	\frac{1}{2} \left< w, R_{1/2}(\varphi) w \right> \ge - \frac{C}{\sqrt{2\pi}} \|w\|_{L^{2}(\mathbb{R})}^{2} = - \widetilde{C} \|w\|_{L^{2}(\mathbb{R})}^{2}.
	\end{equation*}
	
	Next we want to estimate the term stemming from the nonlinearity in terms of $\int_{\mathbb{R}} |D^{1/2}w|^{2} \varphi_{x}\,dx$, and the following results will be of use.
	
	From Theorem 6.5 in \cite{hitchhiker2011} we have for  $s \in (0,1)$ and $p \in [1,\infty)$ such that $sp < n$ there exists a positive constant $C = C(n,p,s)$ such that the Sobolev space $W^{s,p}(\mathbb{R}^{n})$ is continuously embedded in $L^{q}(\mathbb{R}^{n})$ for any $q \in [p, p^*]$ with $p^* := np/(n-sp)$,
	\begin{equation}
	\|f\|_{L^{q}(\mathbb{R}^{n})} \le C \|f\|_{W^{s,p}(\mathbb{R}^{n})},
	\label{embedding}
	\end{equation}
	where $W^{s,2}(\mathbb{R}^{n}) = H^{s}(\mathbb{R}^{n})$.
	
	Next we have the interpolation inequality stated in Proposition 3.1 of \cite{IntNonLinDisp}:
	If $s_{1} \le s \le s_{2}$ with $s = \theta s_{1} + (1-\theta) s_{2}, 0 \le \theta \le 1$, then
	\begin{equation}
	\|u\|_{H^{s}(\mathbb{R}^{n})} \le \|u\|_{H^{s_{1}}(\mathbb{R}^{n})}^{\theta} \|u\|_{H^{s_{2}}(\mathbb{R}^{n})}^{1-\theta}.
	\label{interpolation}
	\end{equation}
	
	Now we turn to the third term on the left-hand side of \eqref{CN_phi} and estimate it as
	\begin{align*}
	\begin{split}
	\int_{\mathbb{R}} w^{3} \varphi_{x}\,dx &\le  \left( \int_{\mathbb{R}} w^{2}\,dx\right)^{1/2} \left( \int_{\mathbb{R}} w^{4} \varphi_{x}^{2}\,dx\right)^{1/2} \\
	&= \left\|w\right\|_{L^{2}(\mathbb{R})} \left\|w \sqrt{\varphi_{x}} \right\|_{L^{4}(\mathbb{R})}^{2} \\
	&\le C \left\|w\right\|_{L^{2}(\mathbb{R})} \left\|w \sqrt{\varphi_{x}} \right\|_{H^{1/4}(\mathbb{R})}^{2} \\
	&\le C \left\|w\right\|_{L^{2}(\mathbb{R})} \left\|w \sqrt{\varphi_{x}} \right\|_{L^{2}(\mathbb{R})} \left\|w \sqrt{\varphi_{x}} \right\|_{H^{1/2}(\mathbb{R})} \\
	&\le \frac{1}{2} \left\|w \sqrt{\varphi_{x}} \right\|_{H^{1/2}(\mathbb{R})}^{2} + \frac{C^{2}}{2} \left\|w\right\|_{L^{2}(\mathbb{R})}^{2} \left\|w \sqrt{\varphi_{x}} \right\|_{L^{2}(\mathbb{R})}^{2} \\
	&\le \frac{1}{2} \left\| D^{1/2} (w \sqrt{\varphi_{x}}) \right\|_{L^{2}(\mathbb{R})}^{2} + \frac{1}{2} \left\|w \sqrt{\varphi_{x}} \right\|_{L^{2}(\mathbb{R})}^{2} \\
	&\quad+ \frac{C^{2}}{2} \left\|w\right\|_{L^{2}(\mathbb{R})}^{2} \left\|w \sqrt{\varphi_{x}} \right\|_{L^{2}(\mathbb{R})}^{2} \\
	&\le \frac{1}{2} \left\| D^{1/2} (w \sqrt{\varphi_{x}}) \right\|_{L^{2}(\mathbb{R})}^{2} + \frac{1}{2} \left(1 + C^{2}\left\|w\right\|_{L^{2}(\mathbb{R})}^{2} \right) \left\|w\right\|_{L^{2}(\mathbb{R})}^{2}.
	\end{split}
	\end{align*}
	The second inequality above is an application of \eqref{embedding} with $n=1$, $p=2$, $s = \frac{1}{4}$, and $q = p^* = 4$, while the third inequality comes from \eqref{interpolation} with $n = 1$ $s = \frac{1}{4}$, $s_{1} = 0$, $s_{2} = \frac{1}{2}$, and $\theta = \frac{1}{2}$.
	For the first term in the last line we set $h = \sqrt{\varphi_{x}}$ and use the commutator identity
	\begin{equation*}
	D^{1/2} h = h D^{1/2} + S_{1/2}(h),
	\end{equation*}
	where $S_{1/2}(h)$ is a remainder operator.
	From Proposition 2.1 in \cite{Ginibre1991} we have
	\begin{equation}
	\vertiii{S_{1/2}(h)} \le \frac{1}{\sqrt{2\pi}} \left\|\widehat{D^{1/2}h}\right\|_{L^{1}(\mathbb{R})},
	\label{commutator_est_2}
	\end{equation}
	where $\vertiii{\cdot}$ denotes the operator norm in $L^{2}(\mathbb{R})$.
	This can be estimated as
	\begin{align*}
	\vertiii{S_{1/2}(h)} &\le \frac{1}{\sqrt{2\pi}} \left\|\widehat{D^{1/2}h}\right\|_{L^{1}(\mathbb{R})} = \frac{1}{\sqrt{2\pi}} \left\||\xi|^{1/2} \hat{h}\right\|_{L^{1}(\mathbb{R})} \\
	&\le \frac{1}{\sqrt{2\pi}} \left\|(1+|\xi|) \hat{h}\right\|_{L^{1}(\mathbb{R})} = \frac{1}{\sqrt{2\pi}} \left( \|\hat{h}\|_{L^{1}(\mathbb{R})} + \| \widehat{h'} \|_{L^{1}(\mathbb{R})} \right).
	\end{align*}
	Noting that $h \in C_{c}^{\infty}(\mathbb{R})$ and in particular $C_{c}^{3}(\mathbb{R})$, so that $h^{(k)} \in L^{1}(\mathbb{R})$ for $k = 0,1,2,3$ and using \eqref{fourierBounded} we have the estimate
	\begin{equation*}
	\vertiii{S_{1/2}(h)} \le \frac{4C}{\sqrt{2\pi}} =: C_{S}.
	\end{equation*}
	Thus, taking the $L^{2}$-norm we obtain
	\begin{align*}
	\begin{split}
	\left\|D^{1/2} (w \sqrt{\varphi_{x}}) \right\|_{L^{2}(\mathbb{R})} &\le \left\|\sqrt{\varphi_{x}} D^{1/2} w \right\|_{L^{2}(\mathbb{R})} + \left\| S_{1/2}(h) w \right\|_{L^{2}(\mathbb{R})} \\
	&\le \left\|\sqrt{\varphi_{x}} D^{1/2} w \right\|_{L^{2}(\mathbb{R})} + C_{S} \left\| w \right\|_{L^{2}(\mathbb{R})},
	\end{split}
	\end{align*}
	for some constant $C_{S}$ depending on $\varphi_{x}$.
	
	Inserting the above estimates in \eqref{CN_phi} we obtain
	\begin{align*}
	\begin{split}
	\frac{1}{2} \left\|u^{n+1} \sqrt{\varphi}\right\|^{2}_{L^{2}(\mathbb{R})} \hspace{-1em}&\hspace{1em}+ \Delta t \left\|\sqrt{\varphi_{x}} D^{1/2} w \right\|_{L^{2}(\mathbb{R})}^{2} - \Delta t \widetilde{C} \|w\|_{L^{2}(\mathbb{R})}^{2} \\
	&\le \frac{1}{2} \left\|u^{n} \sqrt{\varphi}\right\|_{L^{2}(\mathbb{R})}^{2}
	+ \frac{\Delta t}{3} \left\|\sqrt{\varphi_{x}} D^{1/2} w \right\|_{L^{2}(\mathbb{R})}^{2}
	+ \frac{\Delta t}{3} C_{S}^{2} \|w\|_{L^{2}(\mathbb{R})}^{2}\\
	&\quad+ \frac{\Delta t}{6} \left(1 + C^{2} \|w\|_{L^{2}(\mathbb{R})}^{2} \right) \|w\|_{L^{2}(\mathbb{R})}^{2},
	\end{split}
	\end{align*}
	which again implies
	\begin{multline*}
	\frac{1}{2} \left\|u^{n+1} \sqrt{\varphi}\right\|_{L^{2}(\mathbb{R})}^{2} + \frac{2 \Delta t}{3} \int_{\mathbb{R}} \left| D^{1/2} u^{n+1/2} \right|^{2} \varphi_{x}\,dx \\
	\le \frac{1}{2} \left\|u^{n} \sqrt{\varphi}\right\|_{L^{2}(\mathbb{R})}^{2} + C\left(\left\|u^{0}\right\|_{L^{2}(\mathbb{R})}\right) \Delta t,
	\end{multline*}
	where we in the last term have used that the $L^{2}$-norm of $w = u^{n+1/2}$ is bounded by the norm of $u^{0}$.
	By dropping the positive second term on the left-hand side, summing from $n=0$ to $n=m-1$ and utilizing that this is a telescoping sum we obtain
	\begin{equation*}
	\left\|u^{m} \sqrt{\varphi}\right\|_{L^{2}(\mathbb{R})}^{2} \le \left\|u^{0} \sqrt{\varphi}\right\|_{L^{2}(\mathbb{R})}^{2} + C\left(\left\|u^{0}\right\|_{L^{2}(\mathbb{R})}\right) m \Delta t.
	\end{equation*}
	Also, first summing and then dropping $
	\frac{1}{2} \left\|u^{m+1} \sqrt{\varphi}\right\|_{L^{2}(\mathbb{R})}^{2}$ on the left-hand side yields the estimate
	\begin{equation*}
	\Delta t \sum_{n=0}^{m} \int_{-R}^{R} \left| D^{1/2}u^{n+1/2} \right|^{2}\,dx \le \frac{3}{2} \left(\frac{1}{2}\left\|u^{0} \sqrt{\varphi}\right\|_{L^{2}(\mathbb{R})}^{2} + C\left(\left\|u^{0}\right\|_{L^{2}(\mathbb{R})}\right) (m+1) \Delta t \right).
	\end{equation*}
	Together these estimates imply that given initial data $u^{0} \in L^{2}(\mathbb{R})$ we have \[u^{n+1/2} \in \ell^{2}\left([0,m\Delta t]; H^{1/2}([-R,R])\right), \quad 0 \le m \le N,\] which shows that the solutions of the Crank--Nicolson temporal discretized equation also exhibit the local smoothing effect of the BO equation, as the above sequence space is a temporally discrete analogue of $L^{2}\left([0,T);H^{1/2}([-R,R])\right)$.
	Since this smoothing is the main ingredient of the convergence proof in the case of KdV, we want to show that it is present also in our fully discretized element scheme presented in the next section.
	When formulating the scheme we follow \cite{Dutta2016} in using test functions of the form $\varphi v$, where $v$ belongs to some finite element space. This makes \eqref{CN_phi} hold and leads to a $H^{1/2}$-bound like the one obtained here.
	The problem with this form of the scheme is that one loses the a priori preservation of the $L^{2}$-norm that was obtained in \eqref{bound_L2_1} by directly choosing the test function $u^{n+1/2}$.
	To overcome this difficulty we make use of a CFL condition combined with a majorizing differential equation.
	
	\section{Formulation of the discrete scheme}
	\label{scheme}
	Here we formulate the Crank--Nicolson type Galerkin scheme under consideration.
	First we present some remarks on notation and the discretization of time and space.
	Then we use the weak formulation of the problem and a Crank--Nicolson temporal discretization to define a sequence of functions approximating the exact solution at each discrete time step.
	We also define an iteration scheme to solve the implicit equation for each time step and show that this has a solution.
	
	\subsection{Notation and discretization}
	We start by partitioning the real line in equally sized elements in the form of intervals.
	First define the grid points $x_{j} = j\Delta x$ for $j \in \mathbb{Z}$, where $\Delta x$ is the spatial discretization parameter or step length.
	Then the elements can be written as $I_{j} = [x_{j-1}, x_{j}]$. Now turn to the discretization of the time interval considered.
	Given a fixed time horizon $T > 0$ and a temporal discretization parameter $\Delta t$ we set $t_{n} = n\Delta t$ for $n \in \{0,1,\dots,N\}$, where $\left(N+\frac{1}{2}\right)\Delta t = T$.
	For convenience we also use the notation $t_{n+1/2} = \left(t_{n}+t_{n+1}\right)/2$.
	
	Let $\varphi$ be defined as in the previous section.
	We define the weighted $L^{2}$-inner product
	\begin{equation*}
	\left< u, v \right>_{\varphi} := \left< u, v \varphi \right>,
	\end{equation*}
	and the associated weighted norm $\|u\|_{2,\varphi} = \sqrt{\left<u,u\right>_{\varphi}}$.
	
	\subsection{Galerkin scheme}
	As always for the finite element method we start by deriving a weak formulation of the problem \eqref{BO}, like the one obtained in \eqref{weak_1}.
	Applying the Crank--Nicolson temporal discretization to the weak formulation gives \eqref{CN_eqn}.
	Instead of looking for solutions to this equation in $H^{2}(\mathbb{R})$ we will look for solutions belonging to a finite-dimensional subspace $S_{\Delta x}$ of this Hilbert space.
	
	We define the subspace $S_{\Delta x}$ as follows; assuming $r \ge 2$ is a fixed integer we denote the space of polynomials on the interval $I$ of degree not exceeding $r$ by $\mathbb{P}_{r}(I)$.
	Our goal is to find an approximation $u^{\Delta x}$ to the solution of \eqref{BO} which for all $t \in [0,T]$ belongs to
	\begin{equation}
	S_{\Delta x} = \{ v \in H^{2}(\mathbb{R}) \:|\: v \in \mathbb{P}_{r}(I_{j}), j \in \mathbb{Z} \}.
	\label{finite_space}
	\end{equation}
	Now define $\mathcal{P}$ to be the $L^{2}$-orthogonal projection onto $S_{\Delta x}$. Then we define the sequence $\{u^{n}\}_{n \ge 0}$ through the following procedure: Given $u^{0} = \mathcal{P}u_{0}$, find $u^{n+1} \in S_{\Delta x}$ which satisfies
	\begin{equation}
	\left< u^{n+1}, \varphi v \right> - \Delta t \left< \frac{\left(u^{n+1/2}\right)^{2}}{2}, (\varphi v)_{x} \right> + \Delta t \left< \mathcal{H}\left(u^{n+1/2}\right)_{x}, (\varphi v)_{x} \right> = \left< u^{n}, \varphi v \right>,
	\label{CN_element}
	\end{equation}
	for all $v \in S_{\Delta x}$ and $n \in \{0,1,\dots,N\}$.
	Clearly, \eqref{CN_element} is an implicit scheme and consequently one must solve a nonlinear equation to obtain $u^{n+1}$ from $u^{n}$. The procedure for solving this equation at each time step is described in the following subsection.
	Note also that $\left\|u^{0}\right\|_{L^{2}(\mathbb{R})} \le \left\|u_{0}\right\|_{L^{2}(\mathbb{R})}$, and thus from here on we will always use the $L^{2}$-norm of the initial data $u_{0}$ as an upper bound for the norm of the approximation $u^{0}$.
	
	The following \emph{inverse inequalities} presented in \cite[p. 142]{Ciarlet} will be instrumental in our later estimates.
	\begin{equation}
	\|z_{x}\|_{L^{\infty}(\mathbb{R})} \le \frac{C_{1}^{1/2}}{(\Delta x)^{1/2}} \|z_{x}\|_{L^{2}(\mathbb{R})} \le \frac{C_{2}^{1/2}}{(\Delta x)^{3/2}} \|z\|_{L^{2}(\mathbb{R})},\quad z \in S_{\Delta x},
	\label{inverse_ineq}
	\end{equation}
	where the constants $C_{1}, C_{2} > 0$ are independent of $z$ and $\Delta x$.
	Note that the leftmost inequality also holds for $z$ instead of $z_{x}$.
	
	\subsection{Solvability for one time step}
	To show the existence of a solution $u^{n}$ for each time step we define the iteration scheme
	\begin{equation}
	\begin{cases}
	\left< w^{\ell + 1}, \varphi v \right> - \frac{\Delta t}{2} \left< \left( \frac{w^{\ell} + u^{n}}{2} \right)^{2}, (\varphi v)_{x} \right> + \Delta t \left< \left( \mathcal{H}\frac{w^{\ell + 1} + u^{n} }{2} \right)_{x}, (\varphi v)_{x} \right> = \left< u^{n}, \varphi v \right>, \\
	w^{0} = u^{n},
	\end{cases}
	\label{iteration}
	\end{equation}
	which is to hold for all test functions $v \in S_{\Delta x}$.
	The existence of a unique solution $w^{\ell+1}$ to \eqref{iteration} is guaranteed by noting that one may consider this a Galerkin scheme for a linear problem involving a bilinear form in the variables $w^{\ell+1}$ and $v$.
	Using the commutator estimate for the part of the bilinear form involving the Hilbert transform, and choosing $\Delta t$ small enough, $\widetilde{C}\Delta t \le \frac{1}{2}$ say, the bilinear form is coercive with respect to the $L^{2}$-norm, which implies positive definiteness of the resulting matrix system.
	
	We now present a lemma that guarantees the solvability of the implicit scheme \eqref{CN_element}, and the technique for showing this is due to Simon Laumer (private communication).
	\begin{lemma}
		Choose a constant $L$ such that $0 < L < 1$ and set
		\begin{equation*}
		K = \frac{7-L}{1-L} > 7.
		\end{equation*}
		We consider the iteration scheme \eqref{iteration} and assume that the following CFL condition holds,
		\begin{equation}
		\lambda \le \frac{L}{2\sqrt{2} \sqrt{C_{2}} K \|u^{n}\|_{2, \varphi}},
		\label{CFL_1}
		\end{equation}
		where $C_{2}$ is defined in \eqref{inverse_ineq} and $\lambda$ is given by
		\begin{equation}
		\lambda^{2} = \frac{\Delta t^{2}}{\Delta x^{3}}
		\end{equation}
		where $\Delta t$ is taken sufficiently small.
		Then there exists a function $u^{n+1}$ which solves \eqref{CN_element} and $\lim\limits_{\ell \to \infty} w^{\ell} = u^{n+1}$.
		In addition,
		\begin{equation}
		\|u^{n+1}\|_{2, \varphi} \le K \|u^{n}\|_{2, \varphi}.
		\end{equation}
		\label{lemma_onestep}
	\end{lemma}
	
	\begin{proof}[Proof of Lemma \ref{lemma_onestep}]
		We start by rewriting \eqref{iteration} as
		\begin{equation*}
		\left< w^{\ell+1}, \varphi v\right> + \frac{\Delta t}{4} \left< \left(u^{n}w^{\ell}\right)_{x}, \varphi v \right> + \frac{\Delta t}{4} \left< w^{\ell}w_{x}^{\ell}, \varphi v\right> + \frac{\Delta t}{2} \left< \mathcal{H}w^{\ell+1}_{x}, (\varphi v)_{x} \right> = \mathcal{G}(u^{n}, \varphi v)
		\end{equation*}
		for all $v \in S_{\Delta x}$, where
		\begin{equation*}
		\mathcal{G}(u^{n}, \varphi v) := \left< u^{n}, \varphi v \right> + \frac{\Delta t}{8} \left< \left(u^{n}\right)^{2}, (\varphi v)_{x} \right> - \frac{\Delta t}{2} \left< \mathcal{H}u^{n}_{x}, (\varphi v)_{x} \right>.
		\end{equation*}
		From the above equation one derives
		\begin{align*}
		\begin{split}
		\left< w^{\ell+1}-w^{\ell}, \varphi v \right> &+ \frac{\Delta t}{4} \left< \left(u^{n}\left(w^{\ell}-w^{\ell-1}\right)\right)_{x}, \varphi v \right> \\
		&+ \frac{\Delta t}{4} \left< w^{\ell}w^{\ell}_{x}-w^{\ell-1}w^{\ell-1}_{x}, \varphi v \right>
		+ \frac{\Delta t}{2} \left< \mathcal{H}\left(w^{\ell+1}-w^{\ell}\right)_{x}, (\varphi v)_{x} \right> = 0.
		\end{split}
		\end{align*}
		Now substitute $v = w^{\ell+1}-w^{\ell} =: w$ in the above equation to get
		\begin{align*}
		\begin{split}
		\left< w, \varphi w\right> &+ \frac{\Delta t}{2} \left< \mathcal{H}w_{x}, (\varphi w)_{x} \right> \\
		=& \underbrace{-\frac{\Delta t}{4} \left< \left(u^{n}\left(w^{\ell}-w^{\ell-1}\right)\right)_{x}, \varphi w \right>}_{\mathcal{A}_{1}} \underbrace{- \frac{\Delta t}{4} \left< w^{\ell}w^{\ell}_{x} - w^{\ell-1}w^{\ell-1}_{x}, \varphi w \right>}_{\mathcal{A}_{2}}.
		\end{split}
		\end{align*}
		For the term involving the Hilbert transform we estimate as before and use the fact that $\varphi \ge 1$ to obtain
		\begin{equation*}
		\left< \mathcal{H}w_{x}, (\varphi w)_{x} \right> \ge \left\|\sqrt{\varphi_{x}} D^{\frac{1}{2}}w\right\|_{L^{2}(\mathbb{R})}^{2} - \widetilde{C} \|w\|_{L^{2}(\mathbb{R})}^{2} \ge -\widetilde{C} \|w\|_{2,\varphi}^{2}.
		\end{equation*}
		
		We then estimate the term $\mathcal{A}_{2}$ by repeatedly applying Young's inequality and using \eqref{inverse_ineq},
		\begin{align*}
		\begin{split}
		\mathcal{A}_{2} &= \frac{1}{4}\int_{\mathbb{R}}(-\Delta t)\left(w^{\ell}w^{\ell}_{x}-w^{\ell-1}w^{\ell-1}_{x}\right) \varphi w\,dx \\
		&\le \frac{\Delta t^{2}}{8} \int_{\mathcal{R}}\left(w^{\ell}w^{\ell}_{x}-w^{\ell-1}w^{\ell-1}_{x}\right)^{2} \varphi \,dx + \frac{1}{8}\int_{\mathbb{R}}w^{2}\varphi \,dx \\
		&= \frac{\Delta t^{2}}{8} \int_{\mathbb{R}}\left(\left(w^{\ell}-w^{\ell-1}\right)w^{\ell}_{x} + w^{\ell-1}\left(w^{\ell}_{x}-w^{\ell-1}_{x}\right)\right)^{2} \varphi \, dx + \frac{1}{8}\int_{\mathbb{R}}w^{2}\varphi \,dx \\
		&\le \frac{\Delta t^{2}}{4}\int_{\mathbb{R}}\left(w^{\ell}-w^{\ell-1}\right)^{2}\left(w^{\ell}_{x}\right)^{2}\varphi \,dx \\
		&\quad+ \frac{\Delta t^{2}}{4}\int_{\mathbb{R}}\left(w^{\ell-1}\right)^{2}\left(w^{\ell}_{x}-w^{\ell-1}_{x}\right)^{2}\varphi \,dx + \frac{1}{8}\int_{\mathbb{R}}w^{2} \varphi \,dx \\
		&\le \frac{\Delta t^{2}}{4} \left( \|w^{\ell}_{x}\|_{L^{\infty}(\mathbb{R})}^{2} \|w^{\ell}-w^{\ell-1}\|_{2,\varphi}^{2} + \|w^{\ell}_{x}-w^{\ell-1}_{x}\|_{L^{\infty}(\mathbb{R})}^{2} \|w^{\ell-1}\|_{2,\varphi}^{2} \right) + \frac{1}{8} \|w\|_{2,\varphi}^{2} \\
		&\le \frac{C_{2}\Delta t^{2}}{4\Delta x^{3}} \left( \|w^{\ell}\|_{L^{2}(\mathbb{R})}^{2} \|w^{\ell}-w^{\ell-1}\|_{2,\varphi}^{2} + \|w^{\ell}-w^{\ell-1}\|_{L^{2}(\mathbb{R})}^{2} \|w^{\ell-1}\|_{2,\varphi}^{2} \right) + \frac{1}{8} \|w\|_{2,\varphi}^{2},
		\end{split}
		\end{align*}
		which yields
		\begin{equation*}
		\mathcal{A}_{2} \le \frac{1}{8}\|w\|_{2,\varphi}^{2} + \frac{1}{2} C_{2} \lambda^{2} \max\{\|w^{\ell}\|_{2,\varphi}^{2}, \|w^{\ell-1}\|_{2,\varphi}^{2}\} \|w^{\ell}-w^{\ell-1}\|_{2,\varphi}^{2}.
		\end{equation*}
		Likewise we estimate $\mathcal{A}_{1}$,
		\begin{align*}
		\mathcal{A}_{1} \le \frac{\Delta t^{2}}{8} \int_{\mathbb{R}} \left(\left(u^{n}\left(w^{\ell}-w^{\ell-1}\right)\right)_{x}\right)^{2}\varphi \,dx + \frac{1}{8} \int_{\mathbb{R}}w^{2}\varphi \,dx,
		\end{align*}
		where estimates analogous to the preceding ones lead to
		\begin{equation*}
		\mathcal{A}_{1} \le \frac{1}{8} \|w\|_{2,\varphi}^{2} + \frac{1}{2} C_{2} \lambda^{2} \|u^{n}\|_{2,\varphi} \|w^{\ell}-w^{\ell-1}\|_{2,\varphi}^{2}.
		\end{equation*}
		Collecting the bounds we have the following inequality for $\ell \ge 1$,
		\begin{multline*}
		\|w\|_{2,\varphi}^{2} - \frac{\Delta t}{2}\widetilde{C} \|w\|_{2,\varphi}^{2} \\
		\le \frac{1}{4} \|w\|_{2,\varphi}^{2} + C_{2} \lambda^{2} \max\{\|w^{\ell}\|_{2,\varphi}^{2}, \|w^{\ell-1}\|_{2,\varphi}^{2}, \|u^{n}\|_{2,\varphi}^{2}\} \|w^{\ell}-w^{\ell-1}\|_{2,\varphi}^{2},
		\end{multline*}
		which is equivalent to
		\begin{multline*}
		\frac{1}{2}\left(\frac{3}{2}-\widetilde{C}\Delta t\right)\|w^{\ell+1}-w^{\ell}\|_{2,\varphi}^{2} \\
		\le C_{2} \lambda^{2} \max\{\|w^{\ell}\|_{2,\varphi}^{2}, \|w^{\ell-1}\|_{2,\varphi}^{2}, \|u^{n}\|_{2,\varphi}^{2}\} \|w^{\ell}-w^{\ell-1}\|_{2,\varphi}^{2}.
		\end{multline*}
		Assuming $\Delta t$ small enough that $\frac{3}{2}-\widetilde{C}\Delta t \ge 1$ we obtain
		\begin{equation}
		\|w^{\ell+1}-w^{\ell}\|_{2,\varphi}^{2} \le 2 C_{2} \lambda^{2} \max\{\|w^{\ell}\|_{2,\varphi}^{2}, \|w^{\ell-1}\|_{2,\varphi}^{2}, \|u^{n}\|_{2,\varphi}^{2}\} \|w^{\ell}-w^{\ell-1}\|_{2,\varphi}^{2}.
		\label{w_bound}
		\end{equation}
		
		We will now bound $w^{1}$, and so by setting $\ell = 0$ in \eqref{iteration} we get
		\begin{align*}
		\left< w^{1}-u^{n}, \varphi v \right> + \Delta t \left< \mathcal{H}\left( \frac{u^{n}+w^{1}}{2} \right)_{x}, (\varphi v)_{x} \right> &= \frac{\Delta t}{2} \left< \left(u^{n}\right)^{2}, (\varphi v)_{x} \right> \\ &= -\Delta t \left< u^{n} u^{n}_{x}, \varphi v \right>.
		\end{align*}
		Choosing $v = \frac{u^{n}+w^{1}}{2}$ gives
		\begin{multline*}
		\frac{1}{2} \int_{\mathbb{R}} \left(\left(w^{1}\right)^{2} - \left(u^{n}\right)^{2}\right)\varphi \,dx + \Delta t \int_{\mathbb{R}} \mathcal{H}\left(\frac{u^{n}+w^{1}}{2}\right)_{x} \left(\varphi \frac{u^{n}+w^{1}}{2}\right)_{x}\,dx \\
		= -\Delta t \int_{\mathbb{R}} u^{n}u^{n}_{x} \frac{u^{n}+w^{1}}{2}\varphi\,dx.
		\end{multline*}
		Estimating the term involving the Hilbert transform as before, using Young's inequality and \eqref{inverse_ineq} leads to
		\begin{align*}
		\begin{split}
		\left(\frac{1}{4}-\frac{\widetilde{C}\Delta t}{2}\right) \|w^{1}\|^{2}_{2,\varphi} \le \left(\frac{1}{2}+\frac{1}{4}+\frac{\widetilde{C}\Delta t}{2}\right)\|u^{n}\|_{2,\varphi}^{2} + \frac{C_{2}}{2} \lambda^{2} \|u^{n}\|_{2,\varphi}^{4}.
		\end{split}
		\end{align*}
		Choosing $\Delta t$ small enough that $\frac{1}{4}-\frac{\widetilde{C}\Delta t}{2} \ge \frac{1}{8}$ then gives
		\begin{equation}
		\|w^{1}\|^{2}_{2,\varphi} \le 8 \left(1 + C_{2} \lambda^{2}\|u^{n}\|_{2,\varphi}^{2} \right)\|u^{n}\|_{2,\varphi}^{2}.
		\label{w1_bound}
		\end{equation}
		
		Now we claim that the following holds for $\ell \ge 1$,
		\begin{subequations}
			\begin{align}
			\|w^{\ell+1}-w^{\ell}\|_{2,\varphi} &\le L \|w^{\ell}-w^{\ell-1}\|_{2,\varphi}, \label{le_L}\\
			\|w^{\ell}\|_{2,\varphi} &\le K \|u^{n}\|_{2,\varphi}, \label{le_K}\\
			\|w^{1}\|_{2,\varphi} &\le 5 \|u^{n}\|_{2,\varphi}. \label{le_5}
			\end{align}
		\end{subequations}
		The proof follows an induction argument. From \eqref{w1_bound} and \eqref{CFL_1} we get
		\begin{align*}
		\begin{split}
		\|w^{1}\|_{2,\varphi} &\le \left(2\sqrt{2} + 2\sqrt{2}\sqrt{C_{2}} \lambda \|u^{n}\|_{2,\varphi} \right) \|u^{n}\|_{2,\varphi} \\
		&\le \left(2\sqrt{2} + \frac{L}{K} \right) \|u^{n}\|_{2,\varphi} \le 5 \|u^{n}\|_{2,\varphi} \le K \|u^{n}\|_{2,\varphi},
		\end{split}
		\end{align*}
		and so \eqref{le_5} and \eqref{le_K} hold for $\ell = 1$. Setting $\ell = 1$ in \eqref{w_bound} while using \eqref{CFL_1} we obtain
		\begin{align*}
		\begin{split}
		\|w^{2}-w^{1}\|_{2,\varphi} &\le \sqrt{2 C_{2}} \lambda \max \{ \|w^{1}\|_{2,\varphi}, \|u^{n}\|_{2,\varphi} \} \|w^{1}-u^{n}\|_{2,\varphi} \\
		&\le \left(\sqrt{2 C_{2}} \lambda \, 5 \|u^{n}\|_{2,\varphi} \right) \|w^{1}-u^{n}\|_{2,\varphi} \\
		&\le \frac{5 L}{2 K} \|w^{1}-u^{n}\|_{2,\varphi} \le L \|w^{1}-u^{n}\|_{2,\varphi},
		\end{split}
		\end{align*}
		which shows that \eqref{le_L} holds for $\ell = 1$. Now assume that \eqref{le_L} and \eqref{le_K} hold for $\ell = 1,\dots,m$. One then has
		\begin{align*}
		\begin{split}
		\|w^{m+1}\|_{2,\varphi} &\le \sum_{\ell = 0}^{m} \|w^{\ell+1}-w^{\ell}\|_{2,\varphi} + \|w^{0}\|_{2,\varphi} \le \|w^{1}-w^{0}\|_{2,\varphi} \sum_{\ell = 0}^{m} L^{\ell} + \|w^{0}\|_{2,\varphi} \\
		&\le 6\|u^{n}\|_{2,\varphi} \frac{1}{1-L} + \|u^{n}\|_{2,\varphi} = \frac{7-L}{1-L} \|u^{n}\|_{2,\varphi} = K \|u^{n}\|_{2,\varphi},
		\end{split}
		\end{align*}
		thus \eqref{le_K} holds for all $\ell$. This result together with \eqref{w_bound} and \eqref{CFL_1} lead to
		\begin{align*}
		\begin{split}
		\|w^{\ell+1}-w^{\ell}\|_{2,\varphi} &\le \sqrt{2 C_{2}} \lambda \max\{\|w^{\ell}\|_{2,\varphi}, \|w^{\ell-1}\|_{2,\varphi}, \|u^{n}\|_{2,\varphi} \} \|w^{\ell}-w^{\ell-1}\|_{2,\varphi} \\
		&\le \sqrt{2 C_{2}} \lambda K \|u^{n}\|_{2,\varphi} \|w^{\ell}-w^{\ell-1}\|_{2,\varphi} \le L \|w^{\ell}-w^{\ell-1}\|_{2,\varphi}.
		\end{split}
		\end{align*}
		This shows that \eqref{le_L} holds for all $\ell$ as well. Since $0 < L < 1$ this shows that $\{w^{\ell}\}$ is Cauchy and hence converges, which completes the proof of Lemma \ref{lemma_onestep}.
	\end{proof}
	
	\section{Convergence of the scheme}
	\label{convergence}
	In this section we will prove the convergence of the scheme introduced in the previous section.
	As mentioned earlier we will use a local smoothing effect of the BO equation to obtain a $H^{1/2}_{\text{loc}}(\mathbb{R})$ estimate of the approximations.
	We begin with the following important lemma.
	\begin{lemma}
		Let $\lambda$, $K$ and $L$ be defined as in Lemma \ref{lemma_onestep} and let $u^{n}$ be the solution of the scheme \eqref{CN_element}. Assume also that $\Delta t$ satisfies
		\begin{equation}
		\lambda \le \frac{L}{2\sqrt{2} \sqrt{C_{2}} K \sqrt{Y} },
		\label{CFL_2}
		\end{equation}
		for some $Y$ which only depends on $\|u_{0}\|_{L^{2}(\mathbb{R})}$.\footnote{In fact, $Y$ is the solution $y$ of the ordinary differential equation \eqref{ODE} evaluated at the time $T$ defined below.}
		Then there exist a positive time $T$ and a constant C, both depending only on $\|u_{0}\|_{L^{2}(\mathbb{R})}$ such that for all $n$ satisfying $n \Delta t \le T$ the following estimate holds
		\begin{equation}
		\|u^{n}\|_{L^{2}(\mathbb{R})} \le C\left(\|u_{0}\|_{L^{2}(\mathbb{R})}\right).
		\label{un_L2}
		\end{equation}
		In addition, the approximation $u^{n}$ satisfies the following $H^{1/2}$-estimate
		\begin{equation}
		\Delta t \sum_{\left(n+\frac{1}{2}\right)\Delta t \le T} \bigg\lVert D^{1/2}u^{n+1/2} \bigg\rVert_{L^{2}([-R,R])}^{2} \le C \left(\|u_{0}\|_{L^{2}(\mathbb{R})}\right), \quad \left(n+\frac{1}{2}\right) \Delta t < T.
		\label{un_H1/2}
		\end{equation}
		\label{lemma_H1/2_bound}
	\end{lemma}
	
	\begin{proof}[Proof of Lemma \ref{lemma_H1/2_bound}]
		Starting with \eqref{CN_element} it follows that \eqref{CN_phi} holds, and with the same estimates as in the associated section and the fact that $1 \le \varphi(x) \le 2 + 2R$ we obtain the inequality
		\begin{align*}
		\begin{split}
		&\int_{\mathbb{R}}\left(u^{n+1}\right)^{2}\varphi\,dx + 2\Delta t \int_{\mathbb{R}} \left| D^{1/2}u^{n+1/2} \right|^{2} \varphi_{x}\,dx - 2 \Delta t \widetilde{C} \int_{\mathbb{R}}\left(u^{n+1/2}\right)^{2}\varphi\,dx\\
		&\le \int_{\mathbb{R}}\left(u^{n}\right)^{2}\varphi\,dx
		+ \frac{2\Delta t}{3} \int_{\mathbb{R}} \left| D^{1/2}u^{n+1/2} \right|^{2} \varphi_{x}\,dx + \frac{2\Delta t}{3} C_{S}^{2} \int_{\mathbb{R}}\left(u^{n+1/2}\right)^{2}\varphi\,dx \\
		&\quad+ \frac{\Delta t}{3} \int_{\mathbb{R}}\left(u^{n+1/2}\right)^{2}\varphi\,dx + \frac{\Delta t}{3}C^{2} \left(\int_{\mathbb{R}}\left(u^{n+1/2}\right)^{2}\varphi\,dx\right)^{2},
		\end{split}
		\end{align*}
		which again implies
		\begin{align}
		\begin{split}
		&\int_{\mathbb{R}}\left(u^{n+1}\right)^{2}\varphi\,dx + \frac{4}{3} \Delta t \int_{\mathbb{R}} \left| D^{1/2}u^{n+1/2} \right|^{2} \varphi_{x}\,dx \\
		&\le \int_{\mathbb{R}}\left(u^{n}\right)^{2}\varphi\,dx + \Delta t \, C\left[ \int_{\mathbb{R}}\left(u^{n+1/2}\right)^{2}\varphi\,dx + \left(\int_{\mathbb{R}}\left(u^{n+1/2}\right)^{2}\varphi\,dx\right)^{2} \right].
		\end{split}
		\label{stability}
		\end{align}
		Dropping the term involving the fractional derivative and writing $a_{n} = \int_{\mathbb{R}}\left(u^{n}\right)^{2}\varphi\,dx$ then gives
		\begin{equation}
		a_{n+1} \le a_{n} + \Delta t f(a_{n+\frac{1}{2}}),
		\label{diffIneq}
		\end{equation}
		with the function
		\begin{equation*}
		f(a) = C\left[a+a^{2}\right].
		\end{equation*}
		It is easily seen that $a_{n+\frac{1}{2}} \le (a_{n}+a_{n+1})/2$ and so $\{a_{n}\}$ solves the Crank--Nicolson method for the differential inequality
		\begin{equation*}
		\frac{da}{dt} \le f(a). 
		\end{equation*}
		Let us then consider the following ordinary differential equation
		\begin{equation}
		\begin{cases}
		\frac{dy}{dt} = f\left(\frac{K^{2}+1}{2}y\right), & t > 0, \\
		y(0) = a_{0},
		\end{cases}
		\label{ODE}
		\end{equation}
		where $K$ comes from Lemma \ref{lemma_onestep}.
		It is not difficult to show that this differential equation has a unique solution $y$ which blows up at some finite time $T_{\infty}$ depending only on the initial condition, and so we choose $T = T_{\infty}/2$.
		Note that the solution of this ordinary differential equation is strictly increasing and convex.
		We now compare this solution with \eqref{diffIneq} under the assumption that the CFL condition \eqref{CFL_2} holds with $Y := y(T)$.
		We claim that $a_{n} \le y(t_{n})$ for all $n \ge 0$, and argue by induction.
		As $y(0) = a_{0}$, the claim holds for $n=0$.
		Now assume that it holds for $n \in \{0,1,\dots,m\}$. As $0 \le a_{m} \le y(T)$, \eqref{CFL_2} implies that \eqref{CFL_1} holds, and thus Lemma \ref{lemma_onestep} gives $a_{m+1} \le K^{2} a_{m}$.
		Therefore we have
		\begin{equation*}
		a_{m+\frac{1}{2}} \le \frac{1}{2}(a_{m}+a_{m+1}) \le \left(\frac{K^{2}+1}{2}\right)a_{m}.
		\end{equation*}
		The convexity of $f$ then gives
		\begin{align*}
		\begin{split}
		a_{m+1} &\le a_{m} + \Delta t f\left(\frac{K^{2}+1}{2}a_{m}\right) \le y(t_{m}) + \Delta t f\left(\frac{K^{2}+1}{2}y(t_{m})\right) \\
		&\le y(t_{m}) + \Delta t \frac{dy}{dt}(t_{m}) \le y(t_{m+1}),
		\end{split}
		\end{align*}
		which proves the claim.
		Since $\varphi \ge 1$ we get the $L^{2}$-stability estimate \eqref{un_L2},
		\begin{equation*}
		\|u^{n}\|_{L^{2}(\mathbb{R})} \le \sqrt{y(T)} \le C\left(\|u_{0}\|_{L^{2}(\mathbb{R})}\right).
		\end{equation*}
		Consequently, summing over \eqref{stability} yields the estimate
		\begin{equation*}
		\Delta t \sum_{n \Delta t \le T} \int_{-R}^{R} \left|D^{1/2} u^{n+1/2}\right|^{2} \,dx \le C\left(\|u_{0}\|_{L^{2}(\mathbb{R})}\right).
		\end{equation*}
		This proves \eqref{un_H1/2} and completes the proof of Lemma \ref{lemma_H1/2_bound}.
	\end{proof}
	
	\subsection{Bounds on temporal derivative}
	We will here obtain bounds on the temporal derivative to be used later in the analysis. The following lemma will be of use.
	\begin{lemma}
		Let $\psi \in C_{c}^{\infty}\left(-R,R\right)$ and $\varphi$ be defined by properties (a)--(e) in Section \ref{preliminaries}. Then there exists a projection $P : C_{c}^{\infty}\left(-R,R\right) \rightarrow S_{\Delta x} \cap C_{c}\left(-R,R\right)$ such that
		\begin{equation*}
		\int_{\mathbb{R}} u P\left(\psi\right) \varphi \,dx = \int_{\mathbb{R}} u \psi \varphi \, dx, \quad u \in S_{\Delta x}.
		\end{equation*}
		In addition, $P$ satisfies the bounds
		\begin{equation}
		\begin{cases}
		\|P(\psi)\|_{L^{2}(\mathbb{R})} \le C \|\psi\|_{L^{2}(\mathbb{R})}, \\
		\|P(\psi)\|_{H^{1}(\mathbb{R})} \le C \|\psi\|_{H^{1}(\mathbb{R})}, \\
		\|P(\psi)\|_{H^{2}(\mathbb{R})} \le C \|\psi\|_{H^{2}(\mathbb{R})},
		\end{cases}
		\label{projection_ineq}
		\end{equation}
		where the constant $C$ is independent of $\Delta x$.
		\label{lemma_projection}
	\end{lemma}
	\begin{proof}[Proof of Lemma \ref{lemma_projection}]
		The proof is a straightforward adaptation of the $L^{2}$-projection results found in the monograph of Ciarlet \cite[p. 146]{Ciarlet}.
	\end{proof}
	
	In our upcoming estimates we also need the following Sobolev inequality.
	Given $v \in H^{1}(\mathbb{R})$, we have
	\begin{equation}
	\|v\|_{L^{\infty}(\mathbb{R})} \le \|v\|_{H^{1}(\mathbb{R})}.
	\label{L_infty_by_H1}
	\end{equation}
	
	From the definitions of the dual norms in $H^{-2}(\mathbb{R})$ and $H^{-2}([-R,R])$ we have the inequalities
	\begin{equation}
	\int_{-R}^{R} u v\,dx \le \|u\|_{H^{-2}([-R,R])} \|v\|_{H^{2}([-R,R])}
	\label{dual_Holder}
	\end{equation}
	for $u \in H^{-2}([-R,R])$, $v \in H^{2}([-R,R])$, and
	\begin{equation}
	\int_{\mathbb{R}} u v\,dx \le \|u\|_{H^{-2}(\mathbb{R})} \|v\|_{H^{2}(\mathbb{R})}
	\label{dual_Holder_full}
	\end{equation}
	for $u \in H^{-2}(\mathbb{R})$, $v \in H^{2}(\mathbb{R})$.
	
	The above relations together with Lemma \ref{lemma_projection} is used to prove the following lemma regarding the boundedness of the temporal derivatives of the approximate solutions.
	\begin{lemma}
		Let $\{u^{n}\}$ be the solution of the scheme \eqref{CN_element} and assume that the hypothesis of Lemma \ref{lemma_H1/2_bound} holds. Then we have the following estimate
		\begin{align}
		\| D^{+}_{t} u^{n}\varphi\|_{H^{-2}([-R,R])} &\le C\left(\|u_{0}\|_{L^{2}(\mathbb{R})}, R\right),
		\label{un_H-2}
		\end{align}
		where $D^{+}_{t} u^{n}$ is the forward time difference operator
		\begin{equation*}
		D^{+}_{t} u^{n} = \frac{u^{n+1}-u^{n}}{\Delta t}.
		\end{equation*}
		\label{lemma_H-2_bound}
	\end{lemma}
	\begin{proof}[Proof of Lemma \ref{lemma_H-2_bound}]
		Start by rewriting \eqref{CN_element} as
		\begin{equation}
		\left< D^{+}_{t} u^{n}, \varphi v \right> = \left< \frac{\left(u^{n+1/2}\right)^{2}}{2}, (\varphi v)_{x} \right> - \left< \mathcal{H}u^{n+1/2}_{x}, (\varphi v)_{x} \right>,
		\label{CN_temporal}
		\end{equation}
		which holds for all $v \in S_{\Delta x}$. Let $\psi \in C_{c}^{\infty}(-R,R)$ and set $v = P(\psi)$, where $P$ is the projection from Lemma \ref{lemma_projection}, to get
		\begin{equation*}
		\left< D^{+}_{t} u^{n}, \varphi P(\psi) \right> = \left< \frac{\left(u^{n+1/2}\right)^{2}}{2}, (\varphi P(\psi))_{x} \right> - \left< \mathcal{H}u^{n+1/2}_{x}, (\varphi P(\psi))_{x} \right>.
		\end{equation*}
		Using \eqref{L_infty_by_H1}, \eqref{projection_ineq} and \eqref{un_L2} we estimate the first term on the right-hand side as follows
		\begin{align*}
		\begin{split}
		&\int_{\mathbb{R}} \left(u^{n+1/2}\right)^{2} (\varphi P(\psi))_{x} \,dx \\
		&\le \left(\|P(\psi)\|_{L^{\infty}([-R,R])} + \|P(\psi)_{x}\|_{L^{\infty}([-R,R])}(2+2R)\right) \int_{-R}^{R} \left(u^{n+1/2}\right)^{2}\,dx \\
		&\le \left(\|P(\psi)\|_{H^{1}([-R,R])} + \|P(\psi)_{x}\|_{H^{1}([-R,R])}(2+2R)\right) \left\|u^{n+1/2}\right\|_{L^{2}(\mathbb{R})}^{2} \\
		&\le C\left(\|u_{0}\|_{L^{2}(\mathbb{R})}, R\right) \|\psi\|_{H^{2}([-R,R])}.
		\end{split}
		\end{align*}
		The second term can be estimated as
		\begin{align*}
		\begin{split}
		-\int_{\mathbb{R}} \mathcal{H}u^{n+1/2}_{x} (\varphi P(\psi))_{x} \,dx &\le \left\|\mathcal{H}u^{n+1/2}\right\|_{L^{2}(\mathbb{R})} \left\|(\varphi P(\psi))_{xx}\right\|_{L^{2}([-R,R])} \\
		&\le \left\|u^{n+1/2}\right\|_{L^{2}(\mathbb{R})} \|\varphi P(\psi)\|_{H^{2}([-R,R])} \\
		&\le  C\left(\left\|u_{0}\right\|_{L^{2}(\mathbb{R})}, R\right) \|\psi\|_{H^{2}([-R,R])},
		\end{split}
		\end{align*}
		where we have used \eqref{un_L2}, \eqref{projection_ineq} and the $L^{2}$-isometry of the Hilbert transform.
		
		Together this gives
		\begin{equation*}
		\left| \int_{-R}^{R} D^{+}_{t} u^{n} \varphi \psi \,dx \right| = \left| \int_{-R}^{R} D^{+}_{t} u^{n} \varphi P(\psi) \,dx \right| \le C\left(\left\|u_{0}\right\|_{L^{2}(\mathbb{R})}, R\right) \|\psi\|_{H^{2}([-R,R])},
		\end{equation*}
		which implies
		\begin{equation*}
		\left\|D^{+}_{t} u^{n} \varphi\right\|_{H^{-2}([-R,R])} \le C\left(\left\|u_{0}\right\|_{L^{2}(\mathbb{R})}, R\right),
		\end{equation*}
		and the estimate is proven.
		
		If $\psi \in C_{c}^{\infty}(\mathbb{R})$ then $P(\psi) \in S_{\Delta x}$.
		By the exact same arguments as above, but this time on $\mathbb{R}$ instead of $[-R,R]$, we get
		\begin{equation}
		\|D^{+}_{t} u^{n} \varphi \|_{H^{-2}(\mathbb{R})} \le C\left(\left\|u_{0}\right\|_{L^{2}(\mathbb{R})}, R\right).
		\label{un_H-2_full}
		\end{equation}
	\end{proof}
	
	\subsection{Convergence to a weak solution}
	Prior to stating our theorem of convergence we define the weak solution of the Cauchy problem \eqref{BO} in the following way.
	\begin{definition}
		Let $Q > 0$ and $u_{0} \in L^{2}(\mathbb{R})$. Then $u \in L^{2}\left(0,T; H^{1/2}(-Q,Q)\right)$ is a weak solution of \eqref{BO} in the region $(-Q,Q)\times\left[0,T\right)$ if
		\begin{equation}
		\int_{0}^{T} \int_{\mathbb{R}} \left( \phi_{t} u + \phi_{x} \frac{u^{2}}{2} - \left(\mathcal{H}\phi_{xx}\right)u \right) dx\,dt + \int_{\mathbb{R}} \phi(x,0) u_{0}(x)\,dx = 0,
		\label{weak_sol}
		\end{equation}
		for all $\phi \in C_{c}^{\infty}\left( \left(-Q,Q\right)\times\left[0,T\right) \right)$.
	\end{definition}
	
	Now we define the approximate solution $u^{\Delta x} \in S_{\Delta x}$, which will be shown to converge to a weak solution of \eqref{BO}, by the interpolation formula
	\begin{equation}
	u^{\Delta x}(x,t) = 
	\begin{cases}
	u^{n-1/2}(x) + (t-t_{n-1/2}) D^{+}_{t}u^{n-1/2}(x), & t \in [t_{n-1/2}, t_{n+1/2}), n \ge 1, \\
	u^{0}(x) + 2t \frac{u^{1/2}(x)-u^{0}(x)}{\Delta t}, & t \in [0, t_{1/2}).
	\end{cases}
	\label{approx}
	\end{equation}
	We then have the following convergence theorem, which is the main result of the paper.
	\begin{theorem}
		Let $\left\{u^{n}\right\}_{n \in \mathbb{N}}$ be a sequence of functions defined by the scheme \eqref{CN_element} and assume that $\left\|u_{0}\right\|_{L^{2}(\mathbb{R})}$ is finite. Assume furthermore that $\Delta t = \mathcal{O}(\Delta x^{2})$. Then there exist a positive time $T$ and a constant $C$, depending only on $R$ and $\left\|u_{0}\right\|_{L^{2}(\mathbb{R})}$ such that
		\begin{align}
		\left\| u^{\Delta x} \right\|_{L^{\infty}(0,T;L^{2}([-R,R]))} &\le C \left(\left\|u_{0}\right\|_{L^{2}(\mathbb{R})},R\right), \label{bound_L2}\\
		\left\| u^{\Delta x} \right\|_{L^{2}(0,T;H^{1/2}([-R,R]))} &\le C \left(\left\|u_{0}\right\|_{L^{2}(\mathbb{R})},R\right), \label{bound_H1/2}\\
		\left\| \partial_{t} u^{\Delta x} \varphi \right\|_{L^{2}(0,T;H^{-2}([-R,R]))} &\le C \left(\left\|u_{0}\right\|_{L^{2}(\mathbb{R})}, R\right), \label{bound_H-2}
		\end{align}
		where $u^{\Delta x}$ is given by \eqref{approx}. Moreover, there exists a sequence $\{\Delta x_{j}\}_{j=1}^{\infty}$ and a function $u \in L^{2}\left(0,T; L^{2}([-R,R])\right)$ such that
		\begin{equation}
		u^{\Delta x_{j}} \to u \,\,\text{strongly in}\,\, L^{2}\left(0,T; L^{2}([-R,R])\right),
		\label{strong_convergence}
		\end{equation}
		as $\Delta x_{j} \xrightarrow{j \to \infty} 0$. The function $u$ is a weak solution of the Cauchy problem for \eqref{BO}, which is to say that it satisfies \eqref{weak_sol} with $Q = R-1$.
		\label{theorem_convergence}
	\end{theorem}
	
	\begin{remark}
		The convergent subsequence $u^{\Delta x_{j}}$ in Theorem \ref{theorem_convergence} can be used as a constructive proof of existence of solutions to \eqref{BO}, as noted in Section \ref{introduction}.
		On the other hand, owing to the well-posedness for initial data $u_{0} \in L^2(\mathbb{R})$ \cite{Ionescu07} we can in fact conclude that the whole sequence converges as $\Delta x \to 0$.
	\end{remark}
	
	\begin{proof}[Proof of Theorem \ref{theorem_convergence}]
		Assume for simplicity that $T = (N+\frac{1}{2})\Delta t$ for some $N \in \mathbb{N}$.
		For $t \in [t_{n-1/2}, t_{n+1/2})$ we have
		\begin{equation*}
		u^{\Delta x}(x,t) = (1-\alpha_{n}(t)) u^{n-1/2}(x) + \alpha_{n}(t) u^{n+1/2}(x),
		\end{equation*}
		where $\alpha_{n} = (t-t_{n-1/2})/\Delta t \in [0,1)$.
		For $t \in [t_{n-1/2}, t_{n+1/2}), n = 1,2,\dots,N$ one then has
		\begin{align*}
		\begin{split}
		\left\|u^{\Delta x}\right\|_{L^{2}(\mathbb{R})} &\le \|u^{n-1/2}\|_{L^{2}(\mathbb{R})} + \|u^{n+1/2}\|_{L^{2}(\mathbb{R})} \le C\left(\|u_{0}\|_{L^{2}(\mathbb{R})}\right),
		\end{split}
		\end{align*}
		while for $t \in [0, t_{1/2})$ we have
		\begin{equation*}
		\begin{split}
		\left\|u^{\Delta x}\right\|_{L^{2}(\mathbb{R})} \le \|1-(2t)/\Delta t\|_{\infty} \|u^{0}\|_{L^{2}(\mathbb{R})} + \|(2t)/\Delta t\|_{\infty} \|u^{1/2}\|_{L^{2}(\mathbb{R})} \le C\left(\|u_{0}\|_{L^{2}(\mathbb{R})}\right),
		\end{split}
		\end{equation*}
		which proves \eqref{bound_L2}.
		
		Next we have
		\begin{align*}
		\begin{split}
		&\int_{0}^{T} \left\| D^{1/2}u^{\Delta x}\right\|^{2}_{L^{2}([-R,R])}\,dt \\
		&\le 2 \left\|D^{1/2} u^{0}\right\|^{2}_{L^{2}([-R,R])} \int_{0}^{t_{1/2}} \left(1-\frac{2t}{\Delta t}\right)^{2}dt \\
		&\quad+ 2 \left\|D^{1/2}u^{1/2}\right\|^{2}_{L^{2}([-R,R])} \int_{0}^{t_{1/2}} \left(\frac{2t}{\Delta t}\right)^{2}dt \\
		&\quad+ 2 \sum_{n=1}^{N} \left\|D^{1/2} u^{n-1/2}\right\|^{2}_{L^{2}([-R,R])} \int_{t_{n-1/2}}^{t_{n+1/2}} (1-\alpha_{n}(t))^{2} \,dt \\
		&\quad+ 2 \sum_{n=1}^{N} \left\|D^{1/2}u^{n+1/2}\right\|^{2}_{L^{2}([-R,R])} \int_{t_{n-1/2}}^{t_{n+1/2}} (\alpha_{n}(t))^{2} \,dt \\
		&\le \Delta t \,C\left\|u^{0}\right\|^{2}_{H^{1}([-R,R])} + 2 \Delta t \sum_{n=0}^{N} \left\|D^{1/2}u^{n+1/2}\right\|^{2}_{L^{2}([-R,R])} \\
		&= \Delta t \,C\left\|u^{0}\right\|^{2}_{L^{2}([-R,R])} + \Delta t \,C\left\|u^{0}_{x}\right\|^{2}_{L^{2}([-R,R])} + 2 \Delta t \sum_{n=0}^{N} \left\|D^{1/2}u^{n+1/2}\right\|^{2}_{L^{2}([-R,R])},
		\end{split}
		\end{align*}
		where in the last inequality we have used that $\left\|u^{0}\right\|_{H^{1/2}([-R,R])} \le C \left\|u^{0}\right\|_{H^{1}([-R,R])}$ due to Sobolev embedding.
		In the last line we apply \eqref{un_H1/2} to the sum, and the inverse equality \eqref{inverse_ineq} combined with the assumption $\Delta t \le C\Delta x^2$ to the second term to conclude that \eqref{bound_H1/2} holds.
		
		For the third relation, note that
		\begin{align*}
		\begin{split}
		u^{\Delta x}_{t} &= \begin{cases}
		D^{+}_{t} u^{n-1/2}, & (x,t) \in \mathbb{R} \times [t_{n-1/2}, t_{n+1/2}), \\
		\frac{u^{1/2}-u^{0}}{\Delta t/2}, & (x,t) \in \mathbb{R} \times [0, t_{1/2}),
		\end{cases} \\
		&= \begin{cases}
		\frac{D^{+}_{t}u^{n} + D^{+}_{t}u^{n-1}}{2}, & (x,t) \in \mathbb{R} \times [t_{n-1/2}, t_{n+1/2}), \\
		D^{+}_{t} u^{0}, & (x,t) \in \mathbb{R} \times [0, t_{1/2}).
		\end{cases}
		\end{split}
		\end{align*}
		Using \eqref{un_H-2} it is then easy to show that \eqref{bound_H-2} holds.
		If one instead uses \eqref{un_H-2_full} and considers the whole real line $\mathbb{R}$, the same argument leads to
		\begin{equation}
		\|\partial_{t} u^{\Delta x} \varphi\|_{L^{2}(0,T;H^{-2}(\mathbb{R}))} \le C\left( \left\|u_{0}\right\|_{L^{2}(\mathbb{R})}, R \right).
		\label{bound_H-2_full}
		\end{equation}
		
		Also, as $\varphi$ is a positive and bounded smooth function, \eqref{bound_L2} and \eqref{bound_H1/2} give
		\begin{subequations}
			\begin{align}
			\|\varphi u^{\Delta x}\|_{L^{\infty}(0,T;L^{2}([-R,R]))} \le C\left(\|u_{0}\|_{L^{2}(\mathbb{R})}, R\right), \label{bound_L2_phi}\\
			\|\varphi u^{\Delta x}\|_{L^{2}(0,T;H^{1/2}([-R,R]))} \le C\left(\|u_{0}\|_{L^{2}(\mathbb{R})}, R\right). \label{bound_H1/2_phi}
			\end{align}
		\end{subequations}
		
		Based on the bounds \eqref{bound_L2_phi}, \eqref{bound_H1/2_phi} and \eqref{bound_H-2} we apply the Aubin--Simon compactness lemma \cite[Lemma 4.4]{Holden2015} to the set $\{\varphi u^{\Delta x}\}$ to prove the existence of a sequence $\{\Delta x_{j}\}_{j \in \mathbb{N}}$ such that $\Delta x_{j} \to 0$, and a function $\tilde{u}$ such that
		\begin{equation}
		\varphi u^{\Delta x_{j}} \to \tilde{u} \,\, \text{strongly in} \,\, L^{2}(0,T;L^{2}([-R,R])),
		\label{phi_convergence}
		\end{equation}
		as $j$ approaches infinity. As $\varphi \ge 1$, \eqref{phi_convergence} implies that there exists $u$ such that \eqref{strong_convergence} holds.
		The strong convergence allows passage to the limit in the nonlinearity.
		
		Now it remains to prove that $u$ indeed is a weak solution of \eqref{BO}. In the following we consider the standard $L^{2}$-projection of a function $\psi$ with $k+1$ continuous derivatives into the space $S_{\Delta x}$ for some $k \in \mathbb{N}_{0}$, denoted by $\mathcal{P}$.
		That is,
		\begin{equation*}
		\int_{\mathbb{R}} \left(\mathcal{P}\psi(x)-\psi(x)\right)v(x) \,dx = 0, \quad v \in S_{\Delta x}.
		\end{equation*}
		For the above projection we have
		\begin{equation}
		\|\psi(x) - \mathcal{P}\psi(x)\|_{H^{k}(\mathbb{R})} \le C \Delta x \|\psi\|_{H^{k+1}(\mathbb{R})},
		\label{projection_diff}
		\end{equation}
		where the constant $C$ is independent of $\Delta x$. For a proof of the above estimate, see Ciarlet \cite[p. 133]{Ciarlet}.
		
		Note also that from \eqref{CN_temporal} we have for $v \in S_{\Delta x}$ and $n \ge 1$,
		\begin{align*}
		\begin{split}
		\left< D^{+}_{t} u^{n}, \varphi v \right> &- \left< \frac{\left(u^{n+1/2}\right)^{2}}{2}, (\varphi v)_{x} \right> + \left< \mathcal{H}u^{n+1/2}_{x}, (\varphi v)_{x} \right> = 0, \\
		\left< D^{+}_{t} u^{n-1}, \varphi v \right> &- \left< \frac{\left(u^{n-1/2}\right)^{2}}{2}, (\varphi v)_{x} \right> + \left< \mathcal{H}u^{n-1/2}_{x}, (\varphi v)_{x} \right> = 0.
		\end{split}
		\end{align*}
		Averaging the two relations gives
		\begin{align*}
		\begin{split}
		\mathcal{G}_{n}(\varphi v) := \left< D^{+}_{t} u^{n-1/2}, \varphi v \right> &- \frac{1}{2} \left< \frac{\left(u^{n+1/2}\right)^{2} + \left(u^{n-1/2}\right)^{2}}{2}, (\varphi v)_{x} \right> \\
		&+ \left< \mathcal{H}\left( \frac{u^{n+1/2} + u^{n-1/2}}{2} \right)_{x}, (\varphi v)_{x} \right> = 0.
		\end{split}
		\end{align*}
		
		We will start by showing that
		\begin{equation}
		\int_{0}^{T} \int_{\mathbb{R}} \left( u_{t}^{\Delta x} \varphi v - \frac{\left(u^{\Delta x}\right)^{2}}{2} (\varphi v)_{x} - \left(\mathcal{H}u^{\Delta x}\right) (\varphi v)_{xx} \right) dx\,dt = \mathcal{O}(\Delta x),
		\label{weak_sol_2}
		\end{equation}
		for any test function $v \in C_{c}^{\infty}((-R+1,R-1) \times[0,T) )$ and $\varphi$ as defined by properties (a)--(e) in Section \ref{preliminaries}.
		We proceed as follows.
		\begin{align*}
		\begin{split}
		&\int_{0}^{T} \int_{\mathbb{R}} \left( u_{t}^{\Delta x} \varphi v - \frac{\left(u^{\Delta x}\right)^{2}}{2} (\varphi v)_{x} - \left(\mathcal{H}u^{\Delta x}\right) (\varphi v)_{xx} \right) dx\,dt \\
		&= \int_{0}^{t_{1/2}} \int_{\mathbb{R}} \left( u_{t}^{\Delta x} \varphi v - \frac{\left(u^{\Delta x}\right)^{2}}{2} (\varphi v)_{x} - \left(\mathcal{H}u^{\Delta x}\right) (\varphi v)_{xx} \right) dx\,dt \\
		&\quad+ \sum_{n=1}^{N} \int_{t_{n-1/2}}^{t_{n+1/2}} \int_{\mathbb{R}} \left( u_{t}^{\Delta x} \varphi v - \frac{\left(u^{\Delta x}\right)^{2}}{2} (\varphi v)_{x} - \left(\mathcal{H}u^{\Delta x}\right) (\varphi v)_{xx} \right) dx \,dt \\
		&=: \mathcal{C}^{\Delta x} + \mathcal{E}^{\Delta x}.
		\end{split}
		\end{align*}
		Now let $v^{\Delta x} = \mathcal{P}v$ and observe that we may write
		\begin{align*}
		\begin{split}
		\mathcal{C}^{\Delta x} &= \int_{0}^{t_{1/2}} \underbrace{\int_{\mathbb{R}} \left( \left(D^{+}_{t}u^{0}\right)\varphi v^{\Delta x} - \frac{\left(u^{1/2}\right)^{2}}{2} \left(\varphi v^{\Delta x}\right)_{x} + \left(\mathcal{H}u^{1/2}\right)_{x} \left(\varphi v^{\Delta x}\right)_{x} \right) dx}_{\text{$=0$ by \eqref{CN_element}}} dt \\
		&\quad+ \underbrace{\int_{0}^{t_{1/2}} \int_{\mathbb{R}} \left(D^{+}_{t}u^{0}\right) \varphi \left(v - v^{\Delta x}\right)\,dx \,dt}_{\mathcal{C}^{\Delta x}_{1}} \underbrace{- \int_{0}^{t_{1/2}} \int_{\mathbb{R}} \frac{\left(u^{1/2}\right)^{2}}{2} \left(\varphi (v - v^{\Delta x}) \right)_{x} \, dx \,dt}_{\mathcal{C}^{\Delta x}_{2}} \\
		&\quad \underbrace{- \int_{0}^{t_{1/2}} \int_{\mathbb{R}} \left[\frac{\left(u^{\Delta x}\right)^{2}}{2} - \frac{\left(u^{1/2}\right)^{2}}{2}\right] (\varphi v)_{x} \, dx \,dt}_{\mathcal{C}^{\Delta x}_{3}} \\
		&\quad \underbrace{- \int_{0}^{t_{1/2}} \int_{\mathbb{R}} \left(\mathcal{H}u^{1/2}\right) \left(\varphi (v-v^{\Delta x}) \right)_{xx} \,dx\,dt }_{\mathcal{C}^{\Delta x}_{4}} \\
		&\quad \underbrace{- \int_{0}^{t_{1/2}} \int_{\mathbb{R}} \left(\mathcal{H}\left(u^{\Delta x} - u^{1/2}\right)\right) (\varphi v)_{xx} \,dx \,dt }_{\mathcal{C}_{5}^{\Delta x}}, 
		\end{split}
		\end{align*}
		and
		\begin{align*}
		\begin{split}
		\mathcal{E}^{\Delta x} &= \sum_{n=1}^{N} \int_{t_{n-1/2}}^{t_{n+1/2}} \underbrace{\mathcal{G}_{n}(\varphi v^{\Delta x})}_{=0} \,dt + \underbrace{\sum_{n=1}^{N} \int_{t_{n-1/2}}^{t_{n+1/2}} \int_{\mathbb{R}} \left(D^{+}_{t}u^{n-1/2}\right) \varphi \left(v-v^{\Delta x}\right)\,dx\,dt}_{\mathcal{E}^{\Delta x}_{1}} \\
		&\quad \underbrace{ - \sum_{n=1}^{N} \int_{t_{n-1/2}}^{t_{n+1/2}} \int_{\mathbb{R}} \frac{1}{2} \frac{\left(u^{n+1/2}\right)^{2}+\left(u^{n-1/2}\right)^{2}}{2} \left(\varphi (v-v^{\Delta x})\right)_{x} \, dx\,dt}_{\mathcal{E}^{\Delta x}_{2}} \\
		&\quad \underbrace{ -\sum_{n=1}^{N} \int_{t_{n-1/2}}^{t_{n+1/2}} \int_{\mathbb{R}} \frac{1}{2} \left[\left(u^{\Delta x}\right)^{2} - \frac{\left(u^{n+1/2}\right)^{2} + \left(u^{n-1/2}\right)^{2}}{2}\right] (\varphi v)_{x} \,dx\,dt }_{\mathcal{E}^{\Delta x}_{3}} \\
		&\quad \underbrace{- \sum_{n=1}^{N} \int_{t_{n-1/2}}^{t_{n+1/2}} \int_{\mathbb{R}} \left(\mathcal{H}\frac{u^{n+1/2} + u^{n-1/2}}{2}\right) \left(\varphi (v-v^{\Delta x}) \right)_{xx} dx \,dt }_{\mathcal{E}^{\Delta x}_{4}} \\
		&\quad \underbrace{- \sum_{n=1}^{N} \int_{t_{n-1/2}}^{t_{n+1/2}} \int_{\mathbb{R}} \left(\mathcal{H}\left(u^{\Delta x} - \frac{u^{n+1/2} + u^{n-1/2}}{2}\right)\right) (\varphi v)_{xx}\,dx\,dt }_{\mathcal{E}^{\Delta x}_{5}}.
		\end{split}
		\end{align*}
		Now we estimate the preceding terms.
		From \eqref{dual_Holder}, \eqref{projection_diff} and \eqref{bound_H-2} we obtain
		\begin{align*}
		\begin{split}
		\mathcal{C}^{\Delta x}_{1} + \mathcal{E}^{\Delta x}_{1} &= \int_{0}^{T} \int_{-R}^{R} \partial_{t} u^{\Delta x} \varphi \left(v-v^{\Delta x}\right)\,dx\,dt \\
		&\le \int_{0}^{T} \left\| \partial_{t}u^{\Delta x} \varphi \right\|_{H^{-2}([-R,R])} \left\|v-v^{\Delta x}\right\|_{H^{2}([-R+1,R-1])}\,dt \\
		&\le \Delta x \, C\left(\left\|u_{0}\right\|_{L^{2}(\mathbb{R})}, R\right) \left\|v\right\|_{L^{2}(0,T; H^{3}([-R+1,R-1]))} \xrightarrow{\Delta x \to 0} 0.
		\end{split}
		\end{align*}
		
		From \eqref{un_L2}, \eqref{L_infty_by_H1} and \eqref{projection_diff} we get
		\begin{align*}
		\begin{split}
		\mathcal{C}^{\Delta x}_{2} + \mathcal{E}^{\Delta x}_{2} &\le \frac{1}{2} \int_{0}^{t_{1/2}} \int_{-R+1}^{R-1} \left|u^{1/2}\right|^{2} \left|\left(\varphi(v-v^{\Delta x})\right)_{x}\right|\,dx\,dt \\
		&\quad+ \frac{1}{4} \sum_{n=1}^{N} \int_{t_{n-1/2}}^{t_{n+1/2}} \int_{-R+1}^{R-1} \left( \left|u^{n+1/2}\right|^{2} + \left|u^{n-1/2}\right|^{2} \right) \left|\left(\varphi(v-v^{\Delta x})\right)_{x}\right|\,dx\,dt \\
		&\le C(R) \int_{0}^{t_{1/2}} \|v-v^{\Delta x} \|_{H^{2}([-R+1,R-1])} \|u^{1/2}\|_{L^{2}([-R,R])}\,dt \\
		&\quad+ C(R) \sum_{n=1}^{N} \int_{t_{n-1/2}}^{t_{n+1/2}} \|v-v^{\Delta x} \|_{H^{2}([-R+1,R-1])} \|u^{n+1/2}\|_{L^{2}([-R,R])} \,dt \\
		&\quad+ C(R) \sum_{n=1}^{N} \int_{t_{n-1/2}}^{t_{n+1/2}} \|v-v^{\Delta x} \|_{H^{2}([-R+1,R-1])} \|u^{n-1/2}\|_{L^{2}([-R,R])} \,dt \\
		&\le \Delta x\, C\left(\left\|u_{0}\right\|_{L^{2}(\mathbb{R})}, R\right) \left\|v\right\|_{L^{\infty}(0,T; H^{3}([-R+1,R-1]))} \xrightarrow{\Delta x \to 0} 0.
		\end{split}
		\end{align*}
		
		The next terms may be rewritten as
		\begin{align*}
		\begin{split}
		\mathcal{C}_{3}^{\Delta x} + \mathcal{E}^{\Delta x}_{3} &= \frac{\Delta t}{4}\int_{0}^{t_{1/2}} \int_{-R+1}^{R+1} \left( u^{1/2} + u^{0} \right) \left(D^{+}_{t} u^{0}\right) (\varphi v)_{x}\,dx\,dt \\
		&\quad- \int_{0}^{t_{1/2}} \int_{-R+1}^{R+1} \left( u^{0} t \left(D^{+}_{t}u^{0}\right) (\varphi v)_{x} + \frac{1}{2} t^{2} \left(D^{+}_{t}u^{0}\right)^{2} (\varphi v)_{x} \right) dx\,dt \\
		&\quad+ \frac{\Delta t}{4} \sum_{n=1}^{N} \int_{t_{n-1/2}}^{t_{n+1/2}} \int_{-R+1}^{R-1} \left( u^{n+1/2} + u^{n-1/2} \right) \left(D^{+}_{t}u^{n-1/2}\right) (\varphi v)_{x}\,dx\,dt \\
		&\quad- \sum_{n=1}^{N} \int_{t_{n-1/2}}^{t_{n+1/2}} \int_{-R+1}^{R-1} u^{n-1/2} (t-t_{n-1/2}) \left(D^{+}_{t}u^{n-1/2}\right) (\varphi v)_{x} \,dx\,dt \\
		&\quad- \frac{1}{2} \sum_{n=0}^{N} \int_{t_{n-1/2}}^{t_{n+1/2}} \int_{-R+1}^{R-1} (t-t_{n-1/2})^{2} (D^{+}_{t}u^{n-1/2})^{2} (\varphi v)_{x} \,dx\,dt. \\
		\end{split}
		\end{align*}
		Using \eqref{dual_Holder}, \eqref{un_H-2}, \eqref{inverse_ineq} and \eqref{un_L2} we have the estimate
		\begin{align*}
		\begin{split}
		& \int_{0}^{t_{1/2}} \int_{-R+1}^{R-1} u^{0} (D^{+}_{t}u^{0}) (\varphi v)_{x} \,dx\,dt \\
		&\quad+ \sum_{n=1}^{N} \int_{t_{n-1/2}}^{t_{n+1/2}} \int_{-R+1}^{R-1} u^{n-1/2} (D^{+}_{t}u^{n-1/2}) (\varphi v)_{x} \,dx\,dt \\
		&\le \int_{0}^{t_{1/2}} \left\|u^{0}\right\|_{L^{\infty}([-R,R])} \left\| D^{+}_{t}u^{0} \varphi \right\|_{H^{-2}([-R,R])} \left\| \varphi v \right\|_{H^{3}([-R+1,R-1])} \,dt \\
		&\quad+ \sum_{n=1}^{N} \int_{t_{n-1/2}}^{t_{n+1/2}} \left\|u^{n-1/2}\right\|_{L^{\infty}([-R,R])} \left\| D^{+}_{t}u^{n-1/2} \varphi \right\|_{H^{-2}([-R,R])} \left\| \varphi v \right\|_{H^{3}([-R+1,R-1])} dt \\
		&\le C\left(\left\|u_{0}\right\|_{L^{2}(\mathbb{R})}, R\right) \int_{0}^{t_{1/2}} \frac{C}{\sqrt{\Delta x}}\left\|u^{0}\right\|_{L^{2}(\mathbb{R})} \left\| \varphi v \right\|_{H^{3}([-R+1,R-1])} \,dt \\
		&\quad+ C\left(\left\|u_{0}\right\|_{L^{2}(\mathbb{R})}, R\right) \sum_{n=1}^{N} \int_{t_{n-1/2}}^{t_{n+1/2}} \frac{C}{\sqrt{\Delta x}} \left\|u^{n-1/2}\right\|_{L^{2}(\mathbb{R})} \left\| \varphi v \right\|_{H^{3}([-R+1,R-1])} \,dt \\
		&\le \frac{1}{\sqrt{\Delta x}} C\left(\left\|u_{0}\right\|_{L^{2}(\mathbb{R})}, R\right) \|v\|_{L^{\infty}(0,T;H^{3}([-R+1,R-1]))}.
		\end{split}
		\end{align*}
		Similarly we get
		\begin{align*}
		\begin{split}
		& \int_{0}^{t_{1/2}} \int_{-R+1}^{R-1} u^{1/2} (D^{+}_{t}u^{0}) (\varphi v)_{x} \,dx\,dt \\
		&\quad+ \sum_{n=1}^{N} \int_{t_{n-1/2}}^{t_{n+1/2}} \int_{-R+1}^{R-1} u^{n+1/2} (D^{+}_{t}u^{n-1/2}) (\varphi v)_{x} \,dx\,dt \\
		&\le \frac{1}{\sqrt{\Delta x}} C\left(\left\|u_{0}\right\|_{L^{2}(\mathbb{R})}, R\right) \left\|v \right\|_{L^{\infty}(0,T; H^{3}([-R+1,R-1]))}.
		\end{split}
		\end{align*}
		Furthermore we have
		\begin{align*}
		\begin{split}
		&\left| \int_{0}^{t_{1/2}} \int_{-R+1}^{R+1} u^{0} t \left(D^{+}_{t}u^{0}\right) (\varphi v)_{x} \,dx\,dt \right| \\
		&\quad+ \left| \int_{t_{n-1/2}}^{t_{n+1/2}} \int_{-R+1}^{R-1} u^{n-1/2} (t-t_{n-1/2}) \left(D^{+}_{t}u^{n-1/2}\right) (\varphi v)_{x} \,dx\,dt \right| \\
		&\le \frac{\Delta t}{2} \int_{0}^{t_{1/2}} \left\| u^{0} \right\|_{L^{\infty}([-R+1,R-1])} \int_{-R+1}^{R-1} \left| D^{+}_{t}u^{0} \right| |(\varphi v)_{x}| \,dx\,dt \\
		&\quad+ \Delta t \int_{t_{n-1/2}}^{t_{n+1/2}} \left\| u^{n-1/2} \right\|_{L^{\infty}([-R+1,R-1])} \int_{-R+1}^{R-1} \left| D^{+}_{t}u^{n-1/2} \right| |(\varphi v)_{x}| \,dx\,dt,
		\end{split}
		\end{align*}
		and
		\begin{align*}
		\begin{split}
		& \left| \int_{0}^{t_{1/2}} \int_{-R+1}^{R-1} t^{2} \left(D^{+}_{t}u^{0}\right)^{2} (\varphi v)_{x} \,dx\,dt \right| \\
		&\quad+ \left|\int_{t_{n-1/2}}^{t_{n+1/2}} \int_{-R+1}^{R-1} (t-t_{n-1/2})^{2} (D^{+}_{t}u^{n-1/2})^{2} (\varphi v)_{x} \,dx\,dt \right| \\
		&\le \frac{\Delta t}{2} \int_{t_{n-1/2}}^{t_{n+1/2}} \left\| u^{1/2} - u^{0} \right\|_{L^{\infty}([-R+1,R-1])} \int_{-R+1}^{R-1} \left| D^{+}_{t}u^{0} \right| |(\varphi v)_{x}| \,dx\,dt \\
		&\quad+ \Delta t \int_{t_{n-1/2}}^{t_{n+1/2}} \left\| u^{n+1/2} - u^{n-1/2} \right\|_{L^{\infty}([-R+1,R-1])} \int_{-R+1}^{R-1} \left| D^{+}_{t}u^{n-1/2} \right| |(\varphi v)_{x}| \,dx\,dt.
		\end{split}
		\end{align*}
		Thus these terms can be estimated like the preceding ones, and as $\Delta t / \sqrt{\Delta x} = \mathcal{O}(\Delta x)$ we obtain
		\begin{equation*}
		\mathcal{C}_{3}^{\Delta x} + \mathcal{E}^{\Delta x}_{3} \xrightarrow{\Delta x \to 0} 0.
		\end{equation*}
		
		Using the $L^{2}$-isometry of the Hilbert transform, \eqref{un_L2} and \eqref{projection_diff} we obtain
		\begin{align*}
		\begin{split}
		\mathcal{C}^{\Delta x}_{4} + \mathcal{E}^{\Delta x}_{4} &\le \int_{0}^{t_{1/2}} \left\|\mathcal{H}u^{1/2}\right\|_{L^{2}(\mathbb{R})} \left\|\left(\varphi (v-v^{\Delta x})\right)_{xx}\right\|_{L^{2}([-R+1,R-1])} \,dt \\
		&\quad+ \frac{1}{2}\sum_{n=1}^{N} \int_{t_{n-1/2}}^{t_{n+1/2}} \left\|\mathcal{H}u^{n+1/2}\right\|_{L^{2}(\mathbb{R})} \left\|\left(\varphi (v-v^{\Delta x})\right)_{xx}\right\|_{L^{2}([-R+1,R-1])} \,dt \\
		&\quad+ \frac{1}{2}\sum_{n=1}^{N} \int_{t_{n-1/2}}^{t_{n+1/2}} \left\|\mathcal{H}u^{n-1/2}\right\|_{L^{2}(\mathbb{R})} \left\|\left(\varphi (v-v^{\Delta x})\right)_{xx}\right\|_{L^{2}([-R+1,R-1])} \,dt \\
		&\le C\left(\left\|u_{0}\right\|_{L^{2}(\mathbb{R})}, R\right) \int_{0}^{T} \left\|v-v^{\Delta x}\right\|_{H^{2}([-R+1,R-1])} \,dt \\
		&= \Delta x \, C\left(\left\|u_{0}\right\|_{L^{2}(\mathbb{R})}, R\right) \left\|v\right\|_{L^{\infty}(0,T; H^{3}([-R+1,R-1]))} \xrightarrow{\Delta x \to 0} 0.
		\end{split}
		\end{align*}
		
		Finally, from \eqref{dual_Holder_full}, the $H^{2}$-isometry of the Hilbert transform and \eqref{bound_H-2_full} we have the estimate
		\begin{align*}
		\begin{split}
		\mathcal{C}^{\Delta x}_{5} + \mathcal{E}^{\Delta x}_{5} &= \int_{0}^{t_{1/2}} \int_{\mathbb{R}} \Delta t \left(-\frac{1}{2} + \frac{t}{\Delta t}\right) (D^{+}_{t} u^{0}) (\mathcal{H}(\varphi v)_{xx})\,dx\,dt \\
		&\quad+ \sum_{n=1}^{N} \int_{t_{n-1/2}}^{t_{n+1/2}} \int_{\mathbb{R}} \Delta t \left(-\frac{1}{2} + \frac{t-t_{n-1/2}}{\Delta t}\right) (D^{+}_{t}u^{n-1/2}) (\mathcal{H}(\varphi v)_{xx}) \,dx\,dt \\
		&\le \frac{\Delta t}{2} \int_{0}^{T} \int_{\mathbb{R}} \left|\partial_{t}u^{\Delta x} \varphi \right| \left|\mathcal{H}(\varphi v)_{xx}\right|\,dx\,dt \\
		&\le \Delta t \, C\left(\|u_{0}\|_{L^{2}(\mathbb{R})}, R\right) \left\|\varphi v\right\|_{L^{2}(0,T; H^{4}([-R+1,R-1]))} \xrightarrow{\Delta t \to 0} 0.
		\end{split}
		\end{align*}
		
		We combine the preceding estimates to conclude that \eqref{weak_sol_2} holds. Also, observe that by passing $\Delta x \rightarrow 0$ we obtain
		\begin{equation}
		\int_{0}^{T} \int_{\mathbb{R}} \left( u_{t} \varphi v - \frac{u^{2}}{2} (\varphi v)_{x} - \left(\mathcal{H}u\right) (\varphi v)_{xx} \right) dx\,dt = 0,
		\label{weak_phi}
		\end{equation}
		for any test function $v \in C^{\infty}_{c}\left([-R+1,R-1] \times [0,T)\right)$.
		Now we choose $v = \phi / \varphi$ in \eqref{weak_phi} with $\phi \in C^{\infty}_{c}\left([-R+1,R-1]\times [0,T)\right)$ and integrate by parts to conclude that \eqref{weak_sol} holds, that is
		\begin{equation*}
		\int_{0}^{T} \int_{\mathbb{R}} \left( \phi_{t} u + \phi_{x} \frac{u^{2}}{2} - \left(\mathcal{H}\phi_{xx}\right)u \right) dx\,dt + \int_{\mathbb{R}} \phi(x,0) u_{0}(x)\,dx = 0.
		\end{equation*}
		This concludes the proof of Theorem \ref{theorem_convergence}.
	\end{proof}
	
	\section{Numerical experiments}
	\label{experiments}
	Here we present some numerical experiments to illustrate the fully discrete scheme \eqref{CN_element}.
	We stress that this is not an exposition of a numerical method for the problem on the real line in the usual sense as these typically are based on applying a numerical scheme for the periodic version of the problem on a sufficiently large domain such that the reference solutions to are close to zero outside of it, see e.g.\ the experimental sections of \cite{Thomee1998, Dutta2015difference, Dougalis2015}.
	As this paper concerns the convergence of the discretized real line problem only, we find such an approach irrelevant to the main result presented here.
	In our illustration we will also consider a discretized domain which is large enough for the reference solutions to be close to zero outside of it, but we use the full line Hilbert transform where we simply neglect the contributions from outside the domain since the reference solution is nearly zero there.
	We still have to close our system of equations and this we do by imposing a periodic boundary condition by requiring the approximation to take the same value at both ends of the domain.
	A detailed presentation of our setting will follow.
	
	Inspired by \cite{Dutta2016} we define $S_{\Delta x}$ in the following way.
	Let $f$ and $g$ be the functions
	\begin{align*}
	f(y) &= \begin{cases}
	1 + y^2 (2|y|-3), & |y| \le 1, \\
	0, & |y| > 1,
	\end{cases} \\
	g(y) &= \begin{cases}
	y(1-|y|)^2, & |y| \le 1, \\
	0, & |y| > 1.
	\end{cases}
	\end{align*}
	For $j \in \mathbb{Z}$ we define the basis functions
	\begin{align*}
	v_{2j} = f\left(\frac{x-x_{j}}{\Delta x}\right), \quad
	v_{2j+1} = g\left(\frac{x-x_{j}}{\Delta x}\right),
	\end{align*}
	where $x_{j} = j\Delta x$.
	Then $\{v_{j}\}_{-M}^{M}$ spans a $4M+2$ dimensional subspace of $H^{2}(\mathbb{R})$.
	In the following we define $N := 2M$, which is the number of elements used in the approximation.
	We have chosen the weight function $\varphi$ to be $120 + x$ so that it is a positive, increasing function for which the derivative equals 1 on the domain considered, which in this case is the interval $[-100,100]$.
	
	In our experiments we have chosen to set $\Delta t = \mathcal{O}(\Delta x)$ contrary to the assertion $\Delta t = \mathcal{O}(\Delta x^{2})$ from the theory, as smaller time steps did not lead to significant improvement in the convergence rates of the approximations.
	In the iteration \eqref{iteration} to obtain $u^{n+1}$ we chose the stopping condition $\|w^{\ell+1}-w^{\ell}\|_{L^{2}} \le 0.002 \Delta x \|u^{n}\|_{L^{2}}$, which typically required 4--7 iterations for the cruder discretisations and 2--3 iterations for the finer ones.
	On each element, $\left<\mathcal{H}(v_{j})_{x},(v_{i})_{x}\right>$ was evaluated by applying a seven point Gauss--Legendre (GL) quadrature rule to the principal value integral defining $\mathcal{H}(v_{j})_{x}$ in eight GL points, followed by applying the eight point GL rule to the inner product.
	To avoid problems near the singular point, the principal value integral was evaluated using a \textit{subtracted Hilbert transform} method \cite[p. 685]{King2009vol1}.
	
	For $t = n\Delta t$, we set $u^{\Delta x}(x,t) = u^{n}(x,t) = \sum_{j=-M}^{M} u_{j}^{n} v_{j}(x)$.
	We have measured the relative error of the numerical approximation compared to the exact solution $u$,
	\begin{equation*}
	E := \frac{\|u^{\Delta x} - u\|_{L^{2}}}{\|u\|_{L^{2}}},
	\end{equation*}
	where the $L^{2}$-norms were computed with the trapezoidal rule in the points $x_{j}$ of the finest grid under consideration.
	
	As mentioned in the introduction, the BO equation admits an infinite number of conserved quantities \cite{Abdelouhab1989}, and the first three in this hierarchy are
	\begin{align}
	Q_1(u) &= \int_{\mathbb{R}}u\,dx,
	\label{mass} \\
	Q_2(u) &= \int_{\mathbb{R}}\frac{u^2}{2}\,dx,
	\label{momentum} \\
	Q_3(u) &= \int_{\mathbb{R}}\left(\frac{u^3}{3}-u\mathcal{H}u_{x}\right)dx
	\label{energy},
	\end{align}
	typically denoted \emph{mass}, \emph{momentum} and \emph{energy}, respectively.
	Numerical methods which preserve more of the conserved quantities for completely integrable partial differential equations have been observed to give more accurate approximations than the ones preserving fewer.
	Therefore we have computed the relative change $I_{n}(u^{\Delta x}(t)) := (Q_{n}(u^{\Delta x}(t))-Q_{n}(u^{\Delta x}(0)))/Q_{n}(u^{\Delta x}(0))$ for $n = 1,2,3$
	to see how well these quantities were conserved in our discretization.
	We have also included the rate of convergence for $E$, defined for each intermediate step between element numbers $N_1$ and $N_2$ as
	\begin{equation*}
	\frac{\ln(E(N_1))-\ln(E(N_2))}{\ln(N_2)-\ln(N_1)},
	\end{equation*}
	where we now see $E$ as a function of the element number $N$.
	
	In our example we consider a solution to \eqref{BO} presented in \cite{Thomee1998}, namely
	\begin{equation}
	u_{s2}(x,t) = \frac{4c_{1}c_{2} \left( c_{1}\lambda_{1}^{2} + c_{2}\lambda_{2}^{2} + (c_{1}+c_{2})^{2}c_{1}^{-1}c_{2}^{-1}(c_{1}-c_{2})^{-2} \right)}{\left( c_{1}c_{2}\lambda_{1}\lambda_{2}-(c_{1}+c_{2})^{2}(c_{1}-c_{2})^{-2} \right)^{2} + \left( c_{1}\lambda_{1} + c_{2}\lambda_{2} \right)^{2}},
	\label{twoSol}
	\end{equation}
	where $\lambda_{1} := x - c_{1}t - d_{1}$ and $\lambda_{2} := x - c_{2}t - d_{2}$. When $c_{2} > c_{1}$ and $d_{1} > d_{2}$ this expression represents a tall soliton overtaking a smaller one while moving to the right.
	We applied the fully discrete scheme with initial data $u_{0}(x) = u_{s2}(x,0)$ and parameters $c_{1} = 0.3$, $c_{2} = 0.6$, $d_{1} = -30$ and $d_{2} = -55$.
	The time step was set to $\Delta t = 0.5 \|u_{0}\|_{L^{\infty}}^{-1} \Delta x$ and the numerical solutions were computed for $t = 90$ and $t = 180$, that is during and after the taller soliton overtaking the smaller one.
	To approximate the full line problem we set the domain to $[-100,100]$ with the aforementioned periodic boundary condition.
	The results are presented in Table \ref{twosol_norms_Hper} and a comparison between the approximation for $N = 256$ and the exact solution is shown in Figure \ref{twosolBO_per}.
	\begin{figure}[!ht]
		\centering
		\includegraphics[width=4in]{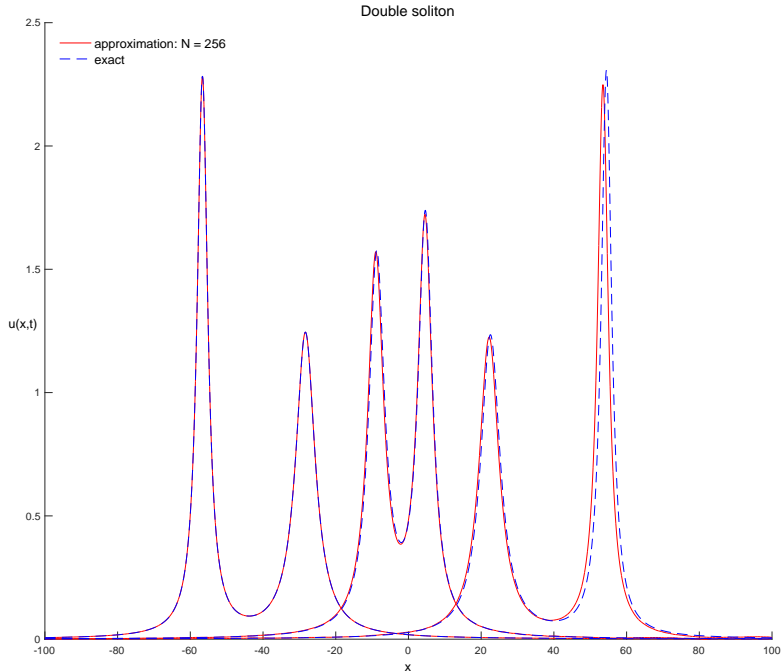}
		\caption{Numerical approximation for $N=256$ and exact solution for $t=0$, $90$ and $180$, respectively positioned from left to right in the plot, for initial data $u_{s2}$.}
		\label{twosolBO_per}
	\end{figure}
	
	\begin{table}[!ht]
		\centering
		\begin{tabular}{ccccccc}
			\toprule
			$t$ & $N$ & $E$ & $\text{rate}_E$ & $I_1$ & $I_2$ & $I_3$ \\
			\midrule
			90 & 128 & 0.01844 & \multirow{2}{*}{-1.45} & \num[round-precision=3,round-mode=figures,scientific-notation=true]{3.788168436539606e-05} & \num[round-precision=3,round-mode=figures,scientific-notation=true]{-7.054039330335533e-04} & \num[round-precision=3,round-mode=figures,scientific-notation=true]{0.005857957549774} \\
			& 256 & 0.05021 & \multirow{2}{*}{1.58} & \num[round-precision=3,round-mode=figures,scientific-notation=true]{-6.607626720852172e-06} & \num[round-precision=3,round-mode=figures,scientific-notation=true]{-0.003852706468442} & \num[round-precision=3,round-mode=figures,scientific-notation=true]{0.016473967969450} \\
			& 512 & 0.01678 & \multirow{2}{*}{0.68} & \num[round-precision=3,round-mode=figures,scientific-notation=true]{-6.644527694261283e-06} & \num[round-precision=3,round-mode=figures,scientific-notation=true]{-9.961675597184954e-04} & \num[round-precision=3,round-mode=figures,scientific-notation=true]{0.004699414586217} \\
			& 1024 & 0.01044 & \multirow{2}{*}{1.16} & \num[round-precision=3,round-mode=figures,scientific-notation=true]{-4.644220296650651e-06} & \num[round-precision=3,round-mode=figures,scientific-notation=true]{3.624401082430920e-04} & \num[round-precision=3,round-mode=figures,scientific-notation=true]{-0.001570622968568} \\
			& 2048 & 0.00467 & \multirow{2}{*}{0.08} & \num[round-precision=3,round-mode=figures,scientific-notation=true]{-3.253423596329705e-06} & \num[round-precision=3,round-mode=figures,scientific-notation=true]{4.156830723175385e-05} & \num[round-precision=3,round-mode=figures,scientific-notation=true]{-5.7061827662940769e-06} \\
			& 4096 & 0.00442 & & \num[round-precision=3,round-mode=figures,scientific-notation=true]{-2.289292275657248e-06} & \num[round-precision=3,round-mode=figures,scientific-notation=true]{4.117044513286984e-06} & \num[round-precision=3,round-mode=figures,scientific-notation=true]{2.073479448870165e-04} \\
			\midrule
			180 & 128 & 0.11959 & \multirow{2}{*}{-1.32} & \num[round-precision=3,round-mode=figures,scientific-notation=true]{1.567825876098693e-04} &  \num[round-precision=3,round-mode=figures,scientific-notation=true]{-6.445923992238868e-04} & \num[round-precision=3,round-mode=figures,scientific-notation=true]{-0.015620294780048} \\
			& 256 & 0.29755 & \multirow{2}{*}{1.75} & \num[round-precision=3,round-mode=figures,scientific-notation=true]{2.481473438135974e-06} & \num[round-precision=3,round-mode=figures,scientific-notation=true]{-0.007795784122829} & \num[round-precision=3,round-mode=figures,scientific-notation=true]{0.033172223688126} \\
			& 512 & 0.08869 & \multirow{2}{*}{0.74} & \num[round-precision=3,round-mode=figures,scientific-notation=true]{-3.758682808565987e-06} & \num[round-precision=3,round-mode=figures,scientific-notation=true]{-0.002415782562417} & \num[round-precision=3,round-mode=figures,scientific-notation=true]{0.011194348955560} \\
			& 1024 & 0.05295 & \multirow{2}{*}{2.35} & \num[round-precision=3,round-mode=figures,scientific-notation=true]{-2.690529846120332e-06} & \num[round-precision=3,round-mode=figures,scientific-notation=true]{9.216569366899723e-04} & \num[round-precision=3,round-mode=figures,scientific-notation=true]{-0.004108082688061} \\
			& 2048 & 0.01040 & \multirow{2}{*}{0.89} & \num[round-precision=3,round-mode=figures,scientific-notation=true]{-1.823173546730924e-06} & \num[round-precision=3,round-mode=figures,scientific-notation=true]{1.161171710542706e-04} & \num[round-precision=3,round-mode=figures,scientific-notation=true]{-4.743854583524921e-04} \\
			& 4096 & 0.00561 & & \num[round-precision=3,round-mode=figures,scientific-notation=true]{-1.264599672777445e-06} &  \num[round-precision=3,round-mode=figures,scientific-notation=true]{1.461678425453621e-05} & \num[round-precision=3,round-mode=figures,scientific-notation=true]{-1.673757454734165e-05} \\
			\bottomrule
		\end{tabular}
		\caption{Relative $L^{2}$-error, $I_1$, $I_2$ and $I_3$ at $t=90$ and $t=180$ for initial data $u_{s2}$.}
		\label{twosol_norms_Hper}
	\end{table}
	
	The latter shows that the shape and position of the numerical approximation for $N = 256$ agrees quite well with the exact solution.
	Still, we observe that for $t = 180$ there is a visible error in the height of the tallest soliton, which has introduced a small phase error in the approximation.
	Since the soliton is very narrow this causes a relative error of 30\%, as seen from Table \ref{twosol_norms_Hper}.
	This error is actually larger than the error for $N = 128$ where the approximate solution has pronounced oscillations.
	Still, the phase error for $N=128$ is much smaller than for $N = 256$, and since the solitons are so narrow this has a much larger impact on the relative $L^{2}$-norm of the error.
	Thus an element number of $N = 128$ or less seems to be too small to get consistent reduction in the errors for increasing $N$.
	We see that for $N$ greater than or equal to 256, the relative $L^{2}$-error at both $t = 90$ and $t = 180$ decreases, but not in a systematic fashion.
	Regarding the conserved quantities, we see that these are preserved quite well in the approximation, and increasing the number of elements tends to improve how well they are preserved, most notably for \eqref{momentum} and \eqref{energy} 
	
	The lack of systematic reduction for the error is most likely caused by our approximation of the full line problem by restricting to a finite interval whilst still using the full line Hilbert transform.
	It should also be pointed out that this is a complicated example, as one has to approximate the nonlinear interaction between two passing solitons.
	Finally, we refer to \cite{proceedings2016} for a study of theoretical convergence rates for this scheme and a modified scheme aimed at the periodic version of \eqref{BO}, given sufficiently regular initial data.
	
	\section*{Acknowledgments} The author is tremendously grateful to Helge Holden and Nils Henrik Risebro for their support and helpful feedback throughout this work.
	Gratitude is also expressed toward the referee for valuable remarks which contributed to improve this paper.
	
	
	\providecommand{\href}[2]{#2}
	\providecommand{\arxiv}[1]{\href{http://arxiv.org/abs/#1}{arXiv:#1}}
	\providecommand{\url}[1]{\texttt{#1}}
	\providecommand{\urlprefix}{URL }

	\medskip
	Received xxxx 20xx; revised xxxx 20xx.
	\medskip
	
\end{document}